\documentclass[12pt]{article}
\usepackage{hyperref}
\usepackage{multirow}
\usepackage{appendix}
\usepackage{amsfonts,amsmath,bm,mathrsfs}
\usepackage{amssymb,amsthm}
\usepackage{graphicx,color}
\usepackage{booktabs}
\usepackage{extarrows}
\usepackage{verbatim}
\usepackage{supertabular}%¿çÒ³±í¸ñ
\usepackage{natbib}
\usepackage{enumerate}
\usepackage{authblk}
\usepackage [latin1]{inputenc}
\makeatletter 
\def\singlespace{\def\baselinestretch{1}\@normalsize}

\DeclareMathOperator*{\diag}{diag}
\DeclareMathOperator*{\RASE}{RASE}
\newcommand{\OR}{\scriptscriptstyle{\mathrm{OR}}}
\theoremstyle{definition}% Theorem-like structures provided by amsthm.sty
\newtheorem{theorem}{Theorem}[section]
\newtheorem{lemma}{Lemma}[section]

\newtheorem{proposition}{Proposition}[section]
\newtheorem{remark}{Remark}[section]
\newtheorem{example}{Example}

\numberwithin{equation}{section}

\begin{document}

\title{Some Theoretical Results Concerning Time-varying Nonparametric Regression with Local Stationary Regressors and Error }     %%%   Main Title of your paper  
\author{Jiyanglin Li}             %%%  1st Author's information   %%%
\author{Tao Li\thanks{Corresponding author. Email: li.tao@mail.shufe.edu.cn}}  %%%  2nd Author's info, if exists, or you may delete this part directly  %%%
\affil{School of Statistics and management, Shanghai University of Finance and Economics, Shanghai, China }
\date{}

\maketitle

\begin{abstract}
With regard to a three-step estimation procedure, proposed without theoretical discussion by Li and You in \emph{Journal of Applied Statistics and Management}, for a nonparametric regression model with time-varying
regression function, local stationary regressors and time-varying AR(p) (tvAR(p)) error process ,
we established all necessary asymptotic properties for each of estimator. We derive the convergence rate and asymptotic normality of the preliminary estimation of nonparametric regression function, establish the asymptotic distribution of time-varying coefficient functions in the error term, and present the asymptotic property of the refined estimation of nonparametric regression function. In addition, with regard to the ULASSO method for variable selection and constant coefficient detection for error term structure, we show that the ULASSO estimator can identify the true error term structure consistently. We conduct two simulation studies to illustrate the finite sample performance of the estimators and validate our theoretical discussion on the properties of the estimators.     \end{abstract}


\section{Introduction}\label{sec:intro}

The nonparametric regression model has played a prominent role in finance and economics analysis due to its capability of capturing the nonlinear relationships  between time series, which are commonly encountered in practice. In some scenario, the factors omitted from the regression model, like those included, are correlated across periods in an unknown way, which will bring errors with autocorrelation. This leads to the regression model
\begin{equation*}
  Y_t=g(\mathbf{X}_t)+e_t, \hspace{1cm} t=1,\cdots,T,
\end{equation*}
where the error process $\{e_t\}$ is autocorrelated but satisfies $E(e_t|\mathbf{X}_t)=0$.

By modeling $\{e_t\}$ as some autocorrelated process, the autocorrelation in the data can be removed so that the regression function can be estimated more efficient. For instance, \cite{xiao2003jasa} considered the case in which $\{e_t\}$ is assumed to be an invertible linear process, with the finite-order ARMA($p,q$) process as the special case, and showed that the autocorrelation function of the error process can improve the estimation of the regression function. \cite{sun2006ecotheo} modeled $\{e_t\}$ as a finite order nonparametric AR process. \cite{liu2010nonparametric} discussed the estimation of regression function based on the models with $\{e_t\}$ following both AR($p$) and  ARMA($p,q$) process.

All aforementioned literatures share the same assumption that $\{\mathbf{X}_t\}$ and $\{e_t\}$ are strictly stationary. In practice, this assumption sometimes is hard to justify since the time series are often observed with trends. One typical approach for applying the time series analysis is to remove the deterministic trend and seasonality components. Various nonparametric detrending procedures have been developed for the time series that contains the smooth trend and ARMA or AR error term, including those of \cite{qiu2013efficient, shao2017oracally, schroder2013adaptive, truong1991nonparametric}, among others. For more detrending methods, the reader is referred to \cite{brockwell2016introduction}. As a result of detrending, the model used is hard to reveal the evolutionary nature of the original data. In recent, an alternative approach for modeling nonstationary time series, viz, the locally stationary time series models come into researchers' view. Locally stationary process is introduced by \cite{dahlhaus1996}, which can be used to model the nonstationary time series directly. Intuitively speaking, a process is locally stationary if over short periods of time (i.e., locally in time) it behaves in an approximately stationary way. Compared to the methods with the (weak) stationary assumption, locally stationary time series model seems more attractive since it can better describe the nonstationary behavior. Due to this fact, several models for locally stationary processes have been proposed in the recent literature. \cite{BeDa2006} discussed the time-varying AR($p$) models with the selection of order $p$. \cite{vogt2012nonparametric} studied nonparametric regression, which includes a wide range of interesting nonlinear time series such as nonparametric autoregressive models. \cite{pei2018nonparametric} developed methods for inference in nonparametric time-varying fixed effects panel data models that allow for locally stationary regressors.

In this paper, we focus on a class of nonparametric regression models with time-varying regression function, local stationary regressors and time-varying AR($p$) (tvAR($p$)) error process, which is studied by \cite{liyouJASM}. The model is given by
\begin{equation}\label{model:ourmodel}
   Y_t=g(t/T, \mathbf{X}_t)+e_t, \qquad e_t-\sum_{i=1}^p\phi_i(t/T)e_{t-i}=\epsilon_{t},\qquad\qquad t=1, \cdots, T,
\end{equation}
where $g(\cdot,\cdot)$ and $\phi_i(\cdot)$ are unknown functions and allowed to change smoothly over time. The $d-$dimension covariates $\mathbf{X}_t$ are assumed to be locally stationary. Without loss of generality, set $d=1$ in this paper. The error process $\{e_t\}$ satisfies $E(e_t|X_t,X_{t-1},\dots)=0$. The white noise $\{\epsilon_t\}$ is independent and identically distributed (i.i.d.) with mean zero and variance $\sigma^2$, and $ E(\epsilon_t|e_{t-1},e_{t-2},\dots,X_t,X_{t-1},\dots)=0$. As discussed by \cite{vogt2012nonparametric}, the tvAR($p$) error process $\{e_t\}$, under some mild conditions, is locally stationary.

Obviously, model \eqref{model:ourmodel} is a quite general form. If $g(\cdot,\cdot)$ and $\phi_i(\cdot)$ are time-invariant, model \eqref{model:ourmodel} is specialized as
\begin{equation}\label{model:timeinvariant}
  Y_t=g(\mathbf{X}_t)+e_t, \qquad e_t-\sum_{i=1}^p\phi_ie_{t-i}=\epsilon_{t},\qquad\qquad t=1, \cdots, T,
\end{equation}
which are discussed in \cite{sun2006ecotheo, liu2010nonparametric}. \cite{liu2010nonparametric} proposed an iterative
estimation procedure for model \eqref{model:timeinvariant} and showed that the iterative estimator is more efficient than the estimator by incorporating the
correlation information of the error into the local linear regression. Motivated by this, Li and You \cite{liyouJASM} defined a three-step estimation procedure for model \eqref{model:ourmodel} and presented its applications in finance through two real data. They first obtained a preliminary estimation for $g(\cdot,\cdot)$ by local linear regression, ignoring the time-varying autoregressive structure of the error term. Then the local linear methods was employed on the residuals from the first step to obtain the estimation of the time-varying autoregressive coefficient functions $\phi_i(\cdot)$. At last, the autocorrelated error term was eliminated by plugging in the estimations obtained from the first two step and a refined estimation of $g(\cdot,\cdot)$ was obtained by using local linear methods again. Intuitively, the refined estimation should be more efficient than the preliminary estimation, which is verified by the real data in \cite{liyouJASM}, since the autocorrelated error term was estimated and eliminated. Unfortunately, no any theoretically results on the estimation was provided by \cite{liyouJASM}. Due to its valuable application, it is worthy to discuss the estimation theoretically. In this paper, we present the asymptotical properties of  preliminary estimation of $g(\cdot,\cdot)$, the estimation of $\phi_i(\cdot)$, and the refined estimation of $g(\cdot,\cdot)$. Moreover, the fact that the refined estimation of $g(\cdot,\cdot)$ is more efficient than the preliminary estimation is proved. This is also illustrated by the simulation studies. For the real data application, the readers are referred to \cite{liyouJASM}.

The rest of the paper is organized as follows. Section \ref{sec:estimator} introduces some notations and presents the estimation method given by \cite{liyouJASM}. Section 3 is the main part of the paper, in which all asymptotical properties of the estimations are provided. At last, the simulation studies are conducted in Section 4.

%%%%%%%%%%%%%%%%%%%%%%%%%%%%%%%%%%%%%%%%%%%%%%%%%%%%%%%%%%%%%%%%%%%%%%%%%%%%%%%%%%%%%%%%%%%%%%%%%%%%%%%%%%%%%%
\section{Notations and estimation procedure}\label{sec:estimator}
In this section, some notations are introduced and the estimation method given by \cite{liyouJASM} are presented briefly for the readers' convenience. The readers are referred to \cite{liyouJASM} for more details.

For model \eqref{model:ourmodel}, we assume, without loss of generality, that the covariate $\mathbf{X}_t=X_t$ is 1-dimensional.

\emph{step 1. The preliminary estimator of $g(\cdot,\cdot)$.} Ignoring the time-varying autoregressive structure of the error term, and by applying local linear fitting method (\cite{fan1996local}) and locally weighted least square estimation method, the preliminary estimator of $g(\cdot,\cdot)$ is derived as:
\begin{equation}\label{eq:preest}
     \begin{pmatrix}
      \hat{g}(u,x) \\
      h\frac{\partial{\hat{g}(u,x)}}{\partial{u}} \\
      h\frac{\partial{\hat{g}(u,x)}}{\partial{x}}
    \end{pmatrix}
  = \big(\mathbf{Z}(u,x)^\top\mathbf{W}(u,x)\mathbf{Z}(u,x)\big)^{-1}\mathbf{Z}(u,x)^\top\mathbf{W}(u,x)\mathbf{Y}.
\end{equation}
where $\mathbf{Y}=(Y_{1},\cdots, Y_{T})^\top$, $\mathbf{1}_T=(1,\cdots,1)^T$ , $\mathbf{U}_u=(\frac{1/T-u}{h}, \cdots,\frac{T/T-u}{h})^\top$, $\mathbf{X}_x=(\frac{X_1-x}{h}, \cdots,$ $\frac{X_T-x}{h})^\top$, $\mathbf{Z}(u,x) = (\mathbf{1}_T, \mathbf{U}_u,\mathbf{X}_x)$.   $\mathbf{W}(u,x)=\diag\Big\{K_h(1/T-u)K_h(X_1-x),\cdots,K_h(T/T-u)K_h(X_T-x)\Big\}$,
$K_h(\cdot)=h^{-1}K(\cdot/h)$ and $K(\cdot): \mathbb{R}\mapsto\mathbb{R}$ is a zero-symmetric kernel function with compact support. It should be noted that the product kernel is applied here and the bandwidths for both time and covariate $X_t$ are assumed to be same. Following the similar procedures in \cite{ruppert1994multivariate} and \cite{pei2018nonparametric}, the results can be easily modified to allow for non-product kernel and different bandwidths.

\emph{step 2. The estimator of $\phi_i(\cdot)$.} Let $\hat{e}_t=Y_t-\hat{g}(t/T, X_t)$ be the estimate of the unobservable error term $e_t$. For the model
$ \hat{e}_t-\sum_{i=1}^p\phi_i\big(t/T\big)\hat{e}_{t-i}\approx\epsilon_{t}$ with the assumption that $p$ is known, by minimizing
\begin{align}
  L(\boldsymbol{\alpha}(u)|\hat{e}, u)&=\sum_{t=p+1}^T\Big\{\hat{e}_t-\sum_{k=1}^p\Big(\phi_{k}(u)+\phi_{k}^{'}(u)(t/T-u)\Big)\hat{e}_{t-k}\Big\}^2K_{h_e}(t/T-u) \label{eqn:least-sqaures},
\end{align}
the estimation of $\boldsymbol\alpha(u)=(\boldsymbol\phi(u)^\top, h_e\boldsymbol\phi'(u)^\top)^{\top}=(\alpha_1(u),\cdots, \alpha_{2p}(u))^\top$ can be obtained as
\begin{align}\label{solution:tvAR}
\hat{\boldsymbol\alpha}(u)&\stackrel{}{=}\begin{pmatrix}
                        \hat{\boldsymbol\phi}(u) \\h_e\hat{\boldsymbol\phi'}(u)
                   \end{pmatrix}
=\begin{pmatrix} \mathbf{I}_p& \mathbf{O}_p \\\mathbf{O}_p & \mathbf{I}_p \end{pmatrix}
(\hat{\mathbf{Z}}_u^\top\mathbf{W}_e\hat{\mathbf{Z}}_u)^{-1}\hat{\mathbf{Z}}_u^\top\mathbf{W}_{e}\hat{\mathbf{e}}_T,
\end{align}
where $\mathbf{I}_p$ is $p\times p$ identity matrix, $\mathbf{O}_p$ is $p\times p$ zero matrix. Let $\hat{\mathbf{e}}_T=(\hat{e}_{p+1},\cdots,\hat{e}_{T})$.
$\hat{\mathbf{Z}}_u=(\hat{\mathbf{e}},~ \mathbf{D}_u\hat{\mathbf{e}})$, and $\hat{\mathbf{e}}=(\hat{\mathbf{e}}_{p},\cdots,\hat{\mathbf{e}}_{T-1})^\top$ with $\hat{\mathbf{e}}_{t-1}=(\hat{e}_{t-1},\hat{e}_{t-2},\cdots,\hat{e}_{t-p})^\top$ for $t=p+1,\cdots,T$, and $\mathbf{D}_u=diag\left\{\frac{(p+1)/T-u}{h_e},~\frac{(p+2)/T-u}{h_e} ,\cdots, \frac{T/T-u}{h_e}\right\}$.
Let $\mathbf{W}_{e}=\diag\Big\{K_{h_e}((p+1)/T-u),K_{h_e}((p+2)/T-u),\cdots,K_{h_e}(T/T-u)\Big\}$,
where $K_{h_e}(\cdot)=h_e^{-1}K(\cdot/h_e)$.

In Section \ref{sec:property}, we show that $\hat{\boldsymbol{\phi}}(u)$ is efficient with the same convergence rate and has the same asymptotic distribution as the estimator when $e_t$ is observable, which is given by \cite{kim2001nonparametric}.

Moreover, \cite{liyouJASM} presented the solution to the problem that the order $p$ of the error series is unknown in the real study.
Denote the true value of $\phi_k(\cdot)$ as $\phi^*_k(\cdot)$. Without loss of generality, it is assumed with $0\leq p_2\leq p_1\leq p$ that the first $p_2$ component of $\boldsymbol{\phi}(\cdot)$ are nonzero smooth functions, while the next $p_1-p_2$ components are nonzero constant, and the final $p-p_1$ components are zeros. Define $\mathcal{S}_1=\{1,\cdots,p_1\}, \mathcal{S}_2=\{1,\cdots,p_2\}$. It is obvious that identifying the coefficient functions and the constant coefficients simultaneously is equivalent to identifying these two sets together. Applying the uniform adaptive LASSO (ULASSO) method proposed by \cite{wang2012parametric}, the ULASSO estimator $\hat{\boldsymbol{\alpha}}_{\lambda, \gamma}(u) = (\hat{\boldsymbol{\phi}}_\lambda(u)^\top, h_e\hat{\boldsymbol{\phi}}^{'}_\gamma(u)^\top)^\top$ for each $u\in(0, 1)$ is defined as the minimizer of the convex function
\begin{align}\label{eqn:penalized-least-squares}
Q_{\lambda,\gamma}(\boldsymbol{\alpha}(u)|\hat{e}, u)&=L(\boldsymbol{\alpha}(u)|\hat{e}, u)
    +\lambda\sum_{k=1}^p\frac{\vert\phi_k(u)\vert}{w_k}+\gamma\sum_{k=1}^p\frac{\vert\phi^{'}_k(u)\vert}{w'_k}, %\\
\end{align}
where $L(\boldsymbol{\alpha}(u)|\hat{e}, u)$ are defined in \eqref{eqn:least-sqaures}, $\lambda\geq0, \gamma\geq0$ are the tuning parameters, $w_k=(\sum_{t=1}^T\tilde{\phi}_k^2(t/T)/T)^{1/2}$ and $w'_k=(\sum_{t=1}^T\tilde{\phi}_k^{'2}(t/T)/T)^{1/2}$ are the uniform adaptive weights for all $u\in(0, 1)$.  Let $\hat{\boldsymbol{\Phi}}_{\lambda, \gamma}=(\hat{\boldsymbol{\alpha}}_{\lambda, \gamma}(1/T),\cdots,\hat{\boldsymbol{\alpha}}_{\lambda, \gamma}(T/T))^\top$, and denote $\hat{\mathcal{S}}_{1\lambda}=\{k: \sum_{t=1}^T\vert\hat{\phi}_{\lambda,k}(t/T)\vert>0\}$, $\hat{\mathcal{S}}_{2\gamma}=\{k: \vert\sum_{t=1}^T\hat{\phi}^{'}_{\gamma,k}(t/T)\vert>0\}$ as the index set of the relative variables and the coefficient functions identified by $\hat\Phi_{\lambda, \gamma}$. The two sets $\hat{\mathcal{S}}_{1\lambda}$ and $\hat{\mathcal{S}}_{2\gamma}$ are
taken as the estimators of $\mathcal{S}_1$ and $\mathcal{S}_2$ respectively.

To determine the tuning parameters, \cite{liyouJASM} proposed a BIC selector:
\begin{align}\label{eqn:BIC-x-unavailable}
BIC_{\lambda, \gamma}(\hat{\boldsymbol{\Phi}}_{\lambda, \gamma}| \hat{e})=\log\{{RSS_{\lambda, \gamma}(\hat{\boldsymbol{\Phi}}_{\lambda, \gamma}| \hat{e})}\}+\widehat{df}_{\lambda, \gamma}\times\frac{\log{(Th_e)}}{Th_e},
\end{align}
where  $\widehat{df}_{\lambda, \gamma}$ is the total number of nonzero elements in both $\hat{\mathcal{S}}_{1\lambda}$ and $\hat{\mathcal{S}}_{2\gamma}$, and
\begin{align*}
{RSS_{\lambda}(\hat{\boldsymbol{\Phi}}_{\lambda, \gamma}| \hat{e})}&= \frac{1}{T^2}\sum_{i=1}^TL(\hat{\boldsymbol{\alpha}}_{\lambda, \gamma}(i/T)|e, i/T).
\end{align*}
Define $(\hat\lambda, \hat\gamma)$ to be the minimizer of \eqref{eqn:BIC-x-unavailable}. Thus, $\hat{\mathcal{S}}_{1\hat{\lambda}}$ and $\hat{\mathcal{S}}_{2\hat{\gamma}}$ are taken as the estimators of $\mathcal{S}_1$ and $\mathcal{S}_2$ respectively.

Once again, there is no any discussions on the property of $\hat{\mathcal{S}}_{1\hat{\lambda}}$ and $\hat{\mathcal{S}}_{2\hat{\gamma}}$ in \cite{liyouJASM}, which is worthwhile and necessary to study. In Section \ref{sec:property}, we show that $(\hat\lambda, \hat\gamma)$ can identify the true model consistently.

\emph{step 3. The refined estimator of $g(\cdot,\cdot)$.} By minimizing the locally weighted least square loss function
\begin{align*}
    &\sum_{t=\widehat{p_0}+1}^T \left\{Y_t-\sum_{k=1}^p\hat{\phi}_k(t/T)\big(Y_{t-k}-\hat{g}((t-k)/T, X_{t-k})-g(u,x)-\frac{\partial g(u,x)}{\partial u}(t/T-u)\right.\\
    &\hspace{3cm} \left. -\frac{\partial g(u,x)}{\partial x}(X_t-x)\right\}^2K_{h^*}(t/T-u)K_{h^*}(X_t-x),
\end{align*}
where $h^*$ is a new bandwidth, the refined estimator of $g(\cdot,\cdot)$ is derived (see\cite{liyouJASM}).
\begin{equation*}
    \begin{pmatrix}
      \hat{g}^*(u,x) \\
      h_e\frac{\partial{\hat{g}^*(u,x)}}{\partial{u}} \\
      h_e\frac{\partial{\hat{g}^*(u,x)}}{\partial{x}}
    \end{pmatrix}
   = \big(\mathbf{Z}^*(u,x)^{\top}\mathbf{W}^*(u,x)\mathbf{Z}^{*}(u,x)\big)^{-1}\mathbf{Z}^{*}(u,x)\mathbf{W}^{*}(u,x)\mathbf{Y}^{*}.
\end{equation*}
where $\hat{\mathbf{Y}}^*=(\hat{Y}^*_{\widehat{p_0}+1},  \cdots, \hat{Y}^*_{T})^\top$ with $\hat{Y}^*_{t}=Y_t-\sum_{k=1}^p\hat{\phi}_k(t/T)\big(Y_{t-k}-\hat{g}((t-k)/T, X_{t-k})$. $\mathbf{Z}^*(u,x) = (\mathbf{1}_{T-\widehat{p_0}}, \mathbf{U}^*_u,\mathbf{X}^*_x)$, $\mathbf{U}^*_u=(\frac{(\widehat{p_0}+1)/T-u}{h}, \cdots, \frac{T/T-u}{h})^\top$, $\mathbf{X}^*_x=(\frac{X_{\widehat{p_0}+1}-x}{h}, \cdots, \frac{X_T-x}{h})^\top$. $\mathbf{W}^*(u,x)=\diag\Big\{K_{h^*}(\frac{\widehat{p_0}+1}{T}-u)K_{h^*}(X_{\widehat{p_0}+1}-x),\cdots,K_{h^*}(\frac{T}{T}-u)K_{h^*}(X_T-x)\Big\}$. Li and You \cite{liyouJASM} explain the reason that the refined estimator is more efficient than the preliminary estimator intuitively. In Section \ref{sec:property}, we present the asymptotic property of the refined estimator and show that the refined estimator is more efficient than the preliminary estimator theoretically.

%%%%%%%%%%%%%%%%%%%%%%%%%%%%%%%%%%%%%%%%%%%%%%%%%%%%%%%%%%%%%%%%%%%%%%%%%%%%%%%%%%%%%%%%%%%%%%%%%%%%%%%%%%%%%%
\section{Main results }\label{sec:property}

\cite{liyouJASM} presented the whole procedure of the estimation and the application to the finance without any discussion on the property of the estimation. Due to its wide range of applications, it is worthwhile and necessary to discuss the statistical property of the estimation. In this article, we present the asymptotic properties of all estimators in \cite{liyouJASM}. In this section, $C>0$ denotes a generic constant that may vary from line to line.

\subsection{Asymptotic results on the preliminary estimator of $g(\cdot,\cdot)$}

For the preliminary estimator of $g(\cdot,\cdot)$, $\hat{g}(u,x)$, we present its uniform convergence rate and asymptotic normality. To do so, we need the following assumptions.
\begin{enumerate}[(C1)]
  \item Both $\{X_t\}$ and $\{e_t\}$ are locally stationary, i.e., for each re-scaled time point $u\in[0, 1]$, there exist two strictly stationary processes $\{X_t(u)\}$ and $\{e_t(u)\}$ such that $\vert X_t-X_t(u)\vert\leq\Big(\Big\vert\frac{t}{T}-u\Big\vert+\frac{1}{T}\Big)U_{t}(u)$ almost surely, and $\vert e_t-e_t(u)\vert\leq\Big(\Big\vert\frac{t}{T}-u\Big\vert+\frac{1}{T}\Big)U^*_{t}(u)$ almost surely,
where $U_t(u)$ is a process of positive variables satisfying $E[(U_t(u))^\rho]\leq C$ for some $\rho>0$ and $C<\infty$, and $U^*_t(u)$ is a process of positive variables satisfying $E[(U^*_t(u))^{\rho^*}]\leq C^*$ for some $\rho^*>0$ and $C^*<\infty$.

  \item $\{X_t, e_t\}$ is $\alpha$-mixing, and the mixing coefficients $\alpha$ satisfies that $\alpha(k)\leq Ak^{-\beta}$ for some $A>0$ and $\beta>\frac{2s-2}{s-2}$ with the same $s$ in (C6).

  \item The density $f(u,x)$ of the variable $X_t(u)$ is smooth in $u$ and bounded away from zero. In particular, $f(u,x)$ is continuously differentiable.

  \item $g(u, x)$ is twice continuously partially differentiable and Lipschitz-continuous.
  \item $K(\cdot)$ has compact support $[-C_1, C_1]$ and is Lipschitz-continuous, $K(x) = K(-x)$, and $\Vert K\Vert_{\infty}=\sup_x\vert K(x)\vert<\infty$.
  \item For some $s>2$, $E|e_t^s|<\infty$.
%  \item Let $f_{X_t}$ and $f_{X_t,X_t+l}$ be the densities of $X_t$ and $(X_t,X_{t+l})$, respectively. For any compact set $S\subset \mathbb{R}$, there exists a constant $C<\infty$ such that $\sup_t\sup_{x\in S}f_{X_t}(x)\leq C$, and $\sup_t\sup_{x\in S}E(|e_t|^s|X_t=x)f_{X_t}(x)\leq C$. Also, there is some $j^*<\infty$ such that for all $j\geq j^*$, $\sup_t\sup_{x,x'\in S}E(|e_te_{t+l}|^s|X_t=x,X_{t+l}=x')f_{X_t,X_{t+l}}(x, x')\leq C$.
  \item With $v_T=\log\log T$, $\beta>\frac{2s-2}{s-2}$ with the same $s$ in (C6), and $\theta=\frac{\beta(1-2/s)-2/s-3}{\beta+1}$, it holds that
      $$  Th^8\to0;\text{~~~~} \frac{v_T\log T}{T^\theta h^2}\to0; \text{~~~~} \frac{1}{T^rh^{r+1}}\to0; \text{~~~~~~as~~~~}T\to\infty.
$$
\end{enumerate}

\begin{remark}(C1) is the basic assumption for a locally stationary process defined as in \cite{vogt2012nonparametric}, which is reasonable in modeling economic and financial data. (C2) is the mixing condition for each time series, which is reasonable and allows the notation in the proofs as simple
as possible. (C3)-(C5) are the regularity conditions commonly used in locally stationary fields and nonparametric settings. (C6) together with (C2) are useful conditions proposed by \cite{hansen2008uniform} to obtain the order of the stochastic part and the variance part. (C7) involves the conditions to satisfy the optimal convergence rate.
\end{remark}

\begin{remark}  The constant $\rho$ can be regarded as a measure of how well $X_t$ is approximated by $X_t(u)$: the larger $\rho$ can be chosen, the less mass is contained in the tails of the distribution of $U_t(u)$. So the approximation of $X_t$ by $X_t(u)$ is getting better for larger $\rho$. This is also true for $\rho^*$ with regard to $e_t$.
\end{remark}

\begin{remark} In general, it is not indispensable for the kernel function to have a bounded support as long as its tails are thin (e.g., a density function that has a second moment). However, as said in \cite{fan2008nonlinear}, 'when the kernel function $K$ has a bounded support, the integration above takes place only around a neighborhood of $x$. Hence, it suffices to assume that the density $f$ has a $p$th continuous derivative at the point $x$.' For the sake of simplicity, it is assumed in this article that the kernel function $K$ has a bounded support. This assumption can be removed at the cost of lengthier arguments.
\end{remark}

The uniform convergence rate and the asymptotic normality of the preliminary estimator are given by Theorem \ref{thm:trend-convergence-rate} and \ref{thm:trend-asymptotic-normality}. The following lemma is needed to prove the theorems.

\begin{lemma}\label{lem:convergence-rate}
Let $S$ be a compact set of $\mathbb{R}$, and the bandwidth $h$ in \eqref{eq:preest} satisfies that
$$
\frac{v_T\log T}{T^\theta h^2}=o(1),
$$
with $v_T=\log\log T$, $\theta=\frac{\beta(1-2/s)-2/s-3}{\beta+1}$, and $\beta>\frac{2s-2}{s-2}$ with the same $s$ in condition (C6).  Then it holds under the conditions (C1)-(C7) that
\begin{equation*}
  \sup_{u\in[0,1],x\in S} \Big|\frac{1}{T}\sum_{t=1}^TK_h(t/T-u)K_h(X_t-x)\big(\frac{t/T-u}{h}\big)^{i}(\frac{X_t-x}{h})^{j}e_t\Big|=O_p\Big(\sqrt{\frac{\log T}{Th^2}}\Big).
\end{equation*}
\end{lemma}

\begin{proof}
Under condition (C2) that $\{e_t\}$ is $\alpha$-mixing, Lemma \ref{lem:convergence-rate} follows immediately from Theorem 4.1 of \cite{vogt2012nonparametric}.
\end{proof}

\begin{theorem}\label{thm:trend-convergence-rate}
Assume that conditions (C1)-(C7) hold. $r=\min\{\rho, 1\}$, where $\rho$ is defined in condition (C1).
Then we have
\begin{align*}
  \sup_{x, u}|\hat{g}(u,x)-g(u,x)| &= O_p\Big(\sqrt{\frac{\log T}{Th^2}}+\frac{1}{T^rh}+h^2\Big)
\end{align*}
\end{theorem}

\begin{proof}
By \eqref{eq:preest}, we have
\begin{align*}
  \hat{g}(u,x)-g(u,x)&=(1, 0, 0)\big(\mathbf{Z}(u,x)^\top\mathbf{W}(u,x)\mathbf{Z}(u,x)\big)^{-1}\mathbf{Z}(u,x)^\top\mathbf{W}(u,x)\mathbf{y}-g(u,x)\\
  &=(1, 0, 0)\big(\mathbf{Z}(u,x)^\top\mathbf{W}(u,x)\mathbf{Z}(u,x)\big)^{-1}\mathbf{Z}(u,x)^\top\mathbf{W}(u,x)\mathbf{G}-g(u,x) \\
  &~~~~+(1, 0, 0)\big(\mathbf{Z}(u,x)^\top\mathbf{W}(u,x)\mathbf{Z}(u,x)\big)^{-1}\mathbf{Z}(u,x)^\top\mathbf{W}(u,x)\mathbf{e} \\
  &\equiv g^B(u,x)+g^V(u,x),
\end{align*}
where
\begin{align*}
g^B(u,x)&=(1, 0, 0)\big(\mathbf{Z}(u,x)^\top\mathbf{W}(u,x)\mathbf{Z}(u,x)\big)^{-1}\mathbf{Z}(u,x)^\top\mathbf{W}(u,x)\mathbf{G}-g(u,x), \\
g^V(u,x)&=(1, 0, 0)\big(\mathbf{Z}(u,x)^\top\mathbf{W}(u,x)\mathbf{Z}(u,x)\big)^{-1}\mathbf{Z}(u,x)^\top\mathbf{W}(u,x)\mathbf{e}.
\end{align*}
By Lemma \ref{lem:convergence-rate}, we have
\begin{equation*}%\label{eqn:var-order}
  \sup_{u\in[0,1],x\in S}|g^V(u,x)|=O_p\Big(\sqrt{\frac{\log T}{Th^2}}\Big).
\end{equation*}
Applying the arguments for Lemma \ref{lem:convergence-rate} to $g^B(u,x)$, we have
\begin{equation*}%\label{eqn:bias-order-1}
  \sup_{u\in[0,1],x\in S}|g^B(u,x)-E(g^B(u,x))|=O_p\Big(\sqrt{\frac{\log T}{Th^2}}\Big).
\end{equation*}
We also claim that
\begin{equation}\label{eqn:bias-order-2}
  \sup_{u\in[0,1],x\in S}|E(g^B(u,x))|=\frac{\mu_2h^2}{2}\Big(\frac{\partial^2{g(u,x)}}{\partial{u}^2}+\frac{\partial^2{g(u,x)}}{\partial{x}^2}\Big)+O\Big(\frac{1}{T^rh}\Big)+o_p(h^2),
\end{equation}
where $r=\min\{\rho, 1\}$. Note that $\sup_{u\in[0,1],x\in S}|E(g^B(u,x))|$ is irrelative with $e_t$, thus \eqref{eqn:bias-order-2} can be proved by following the idea of \cite{vogt2012nonparametric}.

Note that
\begin{align*}
  g^B(u,x)&=(1, 0, 0)\big(\mathbf{Z}(u,x)^\top\mathbf{W}(u,x)\mathbf{Z}(u,x)\big)^{-1}\mathbf{Z}(u,x)^\top\mathbf{W}(u,x)[\mathbf{G}-\mathbf{Z}(u,x)(g(u,x),0,0)^\top] \\
    &= (1, 0, 0)\big(\mathbf{Z}(u,x)^\top\mathbf{W}(u,x)\mathbf{Z}(u,x)\big)^{-1}\mathbf{Z}(u,x)^\top\mathbf{W}(u,x)[\mathbf{G}-g(u,x)\mathbf{1}_T].
\end{align*}
Firstly, using standard results from density estimation, we have
\begin{equation}\label{eqn:ztwz-limit}
  \frac{1}{T}\mathbf{Z}(u,x)^\top\mathbf{W}(u,x)\mathbf{Z}(u,x)=f(u, x)\otimes\diag\{1,\mu_2,\mu_2\}(1+o_p(h^2)),
\end{equation}
where $\otimes$ denotes the kronecker product. Then we consider the term $\mathbf{Z}(u,x)^\top\mathbf{W}(u,x)[\mathbf{G}-g(u,x)\mathbf{1}_T]$. Denote $\bar{K}$ as a Lipschitz-continuous function with support $[-qC_1, qC_1]$ for some $q>1$. Assume that $\bar{K}=1$ for all $x\in[-C_1, C_1]$. Let $\bar{K}_h(\cdot)=\bar{K}(\cdot/h)$, then it holds that
\begin{equation*}
  E\big(\mathbf{Z}(u,x)^\top\mathbf{W}(u,x)[\mathbf{G}-g(u,x)\mathbf{1}_T]\big)=\sum_{i=1}^4\Xi_i(u,x)
\end{equation*}
with
\begin{equation*}
  \Xi_i(u,x)=\frac{1}{T}\sum_{t=1}^TK_h(t/T-u)\xi_i(u,x),~~~~~~~~~~~ i=1,\cdots,4,
\end{equation*}
and
\begin{align*}
  \xi_1(u,x) &= E\Big[\bar{K}_h(X_t-x)\{K_h(X_t-x)-K_h(X_t(t/T)-x)\}\{g(t/T,X_t)-g(u,x)\}\Big], \\
  \xi_2(u,x) &= E\Big[\bar{K}_h(X_t-x)K_h(X_t(t/T)-x)\{g(t/T,X_t)-g(t/T,X_t(t/T))\} \Big], \\
  \xi_3(u,x) &= E\Big[\{\bar{K}_h(X_t-x)-\bar{K}_h(X_t(t/T)-x)\}K_h(X_t(t/T)-x) \\
  &~~~~\times\{g(t/T,X_t(t/T))-g(u,x)\}\Big], \\
  \xi_4(u,x) &= E\Big[K_h(X_t(t/T)-x)\{g(t/T,X_t(t/T))-g(u,x)\} \Big].
\end{align*}
We firstly consider $\Xi_1(u,x)$. Since the kernel function $K(\cdot)$ is bounded, we can find a constant $C\leq\infty$ such that $|K(x)-K(x')|\leq C|K(x)-K(x')|^r$ for $r=\min\{\rho, 1\}$. By the definition of $\bar{K}$ and the smoothness of $g(u,x)$, $\bar{K}_h(X_t-x)|g(t/T,X_t)-g(u,x)|$ can be bounded by $Ch$. Thus, it follows with conditions (C1) and (C5) that
\begin{align*}
  \Xi_1(u,x) &\leq \frac{Ch}{T}\sum_{t=1}^TK_h(t/T-u)E(|K_h(X_t-x)-K_h(X_t(t/T)-x)|) \\
  &\leq \frac{C}{T}\sum_{t=1}^TK_h(t/T-u)E\Big(\Big|K\Big(\frac{X_t-x}{h}\Big)-K\Big(\frac{X_t(t/T)-x}{h}\Big)\Big|^r\Big) \\
  &\leq \frac{C}{T}\sum_{t=1}^TK_h(t/T-u)E\Big(\Big|\frac{1}{Th}U_t(t/T)\Big|^r\Big) \\
  &\leq \frac{C}{T^rh^r}
\end{align*}
uniformly in $u$ and $x$. Using similarly arguments, we can find that $\sup_{u\in[0,1],x\in S}|\Xi_2(u,x)|\leq\frac{C}{T^rh}$ and $\sup_{u\in[0,1],x\in S}|\Xi_3(u,x)|\leq\frac{C}{T^rh^r}$. Finally, applying Lemma \ref{lem:convergence-rate} and conditions (C3) and (C4), we obtain that
\begin{equation*}
  \sup_{u\in[0,1],x\in S}|\Xi_4(u,x)|=\frac{\mu_2h^2}{2}\Big(\frac{\partial^2{g(u,x)}}{\partial{u}^2}+\frac{\partial^2{g(u,x)}}{\partial{x}^2}\Big)+o_p(h^2).
\end{equation*}
Thus \eqref{eqn:bias-order-2} holds by combining the results of $\Xi_1(u,x),\cdots,\Xi_4(u,x)$ and \eqref{eqn:ztwz-limit}. Then we have
\begin{align*}
  &~~~~~\sup_{u\in[0,1],x\in S}|\hat{g}(u,x)-g(u,x)| \\
  &=\sup_{u\in[0,1],x\in S}|\hat{g}^V(u,x)+\hat{g}^B(u,x)-E(\hat{g}^B(u,x))+E(\hat{g}^B(u,x))| \\
  &\leq\sup_{u\in[0,1],x\in S}|\hat{g}^V(u,x)|+\sup_{u\in[0,1],x\in S}|\hat{g}^B(u,x)-E(\hat{g}^B(u,x))|+\sup_{u\in[0,1],x\in S}|E(\hat{g}^B(u,x))|\\
  &=O_p\Big(\sqrt{\frac{\log(T)}{Th^2}}+\frac{1}{T^rh}+h^2\Big).
\end{align*}
\end{proof}

\begin{theorem}\label{thm:trend-asymptotic-normality}
Under the conditions (C1)-(C7), let $r=\min\{\rho, 1\}>\frac{1}{2}$, which guarantees that the bandwidth can be selected to obtain the optimal convergence rate. Then for any $u\in(0, 1)$,
\begin{eqnarray*}
\sqrt{Th^2}
    \begin{bmatrix}
      \begin{pmatrix}
      \hat{g}(u,x) \\
      h\frac{\partial{\hat{g}(u,x)}}{\partial{u}} \\
      h\frac{\partial{\hat{g}(u,x)}}{\partial{x}}
      \end{pmatrix}
      -
      \begin{pmatrix}
      {g}(u,x) \\
      h\frac{\partial{{g}(u,x)}}{\partial{u}} \\
      h\frac{\partial{{g}(u,x)}}{\partial{x}}
      \end{pmatrix}
    -
    \begin{pmatrix}
    \frac{\mu_2h^2}{2}\Big(\frac{\partial^2{g(u,x)}}{\partial^2{u}}+\frac{\partial^2{g(u,x)}}{\partial^2{x}}\Big) \\
    0 \\
    0
    \end{pmatrix}+o_p(h^2)
   \end{bmatrix}&\stackrel{D}{\longrightarrow}{D}~N(\mathbf{0}, \mathbf{V}_{u,x}),\\
    &\mathrm{as}~~ T\to\infty,
\end{eqnarray*}
where  $\mathbf{V}_{u,x}=(\frac{\int_0^1\gamma_0(u)du}{f(u, x)}\big)\boldsymbol{\Sigma}$, and
\begin{equation}\label{eqn:Sigma}
\boldsymbol{\Sigma}=
      \begin{pmatrix}
        \nu_0^2 & \frac{\nu_0\nu_1}{\mu_2} & \frac{\nu_0\nu_1}{\mu_2} \\
        \frac{\nu_0\nu_1}{\mu_2} & \frac{\nu_0\nu_2}{\mu_2^2} & \frac{\nu_1^2}{\mu_2^2} \\
        \frac{\nu_0\nu_1}{\mu_2} & \frac{\nu_1^2}{\mu_2^2} & \frac{\nu_0\nu_2}{\mu_2^2} \\
      \end{pmatrix}.
\end{equation}
\end{theorem}

To prove Theorem \ref{thm:trend-asymptotic-normality}, we define the following notations. $\mathbf{X}^*_x=(\frac{X_1(1/T)-x}{h}, \cdots, \frac{X_T(1/T)-x}{h})^\top$, $\mathbf{Z}^*(u,x) = (\mathbf{1}_T, \mathbf{U}_u,\mathbf{X}^*_x)$,
$\mathbf{W}^*(u,x)=\diag\Big\{K_h(1/T-u)K_h(X_1(1/T)-x),\cdots,K_h(T/T-u)K_h(X_T(1/T)-x)\Big\}$,
$\mathbf{R}^*(u,x)=(\mathbf{U}_u\odot\mathbf{U}_u, \mathbf{X}^*_x\odot\mathbf{X}^*_x, \mathbf{U}_u\odot\mathbf{X}^*_x)$. Let $\mathbf{Z}^*_t(u,x)$ and $\mathbf{R}^*_t(u,x)$ be the $t$-th row of $\mathbf{Z}^*(u,x)$ and $\mathbf{R}^*(u,x)$, respectively. Denote $\mathbf{H}_g(u,x)$ the Hessian matrix of $g$ at $(u,x)$, and $\mathbf{Q}_g(u,x)$ a $T$-dimensional vector whose $t$-th element is $(t/T-u, X_t(t/T)-x)^\top\mathbf{H}_g(u,x)(t/T-u, X_t(t/T)-x)$. Define \begin{align}\label{eqn:Gamma}
\boldsymbol\Gamma(u)&=\begin{pmatrix}
                         \gamma_0(u) & \gamma_1(u) & \gamma_2(u) & \cdots & \gamma_{p-1}(u) \\
                         \gamma_1(u) & \gamma_0(u) & \gamma_1(u) & \cdots & \gamma_{p-2}(u) \\
                         \vdots & \vdots & \vdots & \ddots & \vdots \\
                         \gamma_{p-1}(u) & \gamma_{p-2}(u) & \gamma_{p-3}(u) & \cdots & \gamma_{0}(u)
                      \end{pmatrix},
\end{align}
where $\gamma_k(u)=\sum_{j=0}^\infty\varphi_j(u)\varphi_{j+k}(u)$ is the $k$-th order auto-covariance function of the approximate stationary process $\{e_t(u)\}$ for the rescaled time point $u\in(0, 1)$.

\begin{proof}
With $\hat{g}^V(u,x)$ and $\hat{g}^B(u,x)$ as in the proof of Theorem \ref{thm:trend-convergence-rate}, we have
\begin{align*}
&~~~~~\sqrt{Th^2}
    \begin{bmatrix}
      \begin{pmatrix}
      \hat{g}(u,x) \\
      h\frac{\partial{\hat{g}(u,x)}}{\partial{u}} \\
      h\frac{\partial{\hat{g}(u,x)}}{\partial{x}}
      \end{pmatrix}
      -
      \begin{pmatrix}
      {g}(u,x) \\
      h\frac{\partial{{g}(u,x)}}{\partial{u}} \\
      h\frac{\partial{{g}(u,x)}}{\partial{x}}
      \end{pmatrix}
    \end{bmatrix} \\
&= \sqrt{Th^2}
    \begin{Bmatrix}
      \big(\mathbf{Z}(u,x)^\top\mathbf{W}(u,x)\mathbf{Z}(u,x)\big)^{-1}\mathbf{Z}(u,x)^\top\mathbf{W}(u,x)\mathbf{G}
      -
      \begin{pmatrix}
      {g}(u,x) \\
      h\frac{\partial{{g}(u,x)}}{\partial{u}} \\
      h\frac{\partial{{g}(u,x)}}{\partial{x}}
      \end{pmatrix}
    \end{Bmatrix} \\
&\hspace{2cm}~~~~+\sqrt{Th^2}\big(\mathbf{Z}(u,x)^\top\mathbf{W}(u,x)\mathbf{Z}(u,x)\big)^{-1}\mathbf{Z}(u,x)^\top\mathbf{W}(u,x)\mathbf{e} \\
&\stackrel{\vartriangle}{=}\sqrt{Th^2}B(u,x)+V(u,x),
\end{align*}
where
\begin{align*}
  B(u,x) &=
      \big(\mathbf{Z}(u,x)^\top\mathbf{W}(u,x)\mathbf{Z}(u,x)\big)^{-1}\mathbf{Z}(u,x)^\top\mathbf{W}(u,x)\mathbf{G}
      -
      \begin{pmatrix}
      {g}(u,x) \\
      h\frac{\partial{{g}(u,x)}}{\partial{u}} \\
      h\frac{\partial{{g}(u,x)}}{\partial{x}}
      \end{pmatrix}, \\
  V(u,x) &= \sqrt{Th^2}\big(\mathbf{Z}(u,x)^\top\mathbf{W}(u,x)\mathbf{Z}(u,x)\big)^{-1}\mathbf{Z}(u,x)^\top\mathbf{W}(u,x)\mathbf{e}.
\end{align*}
Henceforth, we refer to $B(u,x)$ and $V(u,x)$ as the bias part and the variance part, respectively.

By conditions (C1) and (C4), we have
\begin{align*}
  \Big|g(t/T,X_t(t/T))-g(t/T,X_t)\Big| &\leq C\Big|X_t(t/T)-X_t\Big|\leq\frac{C}{T}U_t(t/T) =o_p(1).
\end{align*}
Hence, $g(t/T,X_t(t/T))=g(t/T,X_t)+o_p(1)$. By Taylor Theorem, we have
\begin{align*}
  \mathbf{G} &= \mathbf{Z}^*(u,x)\begin{pmatrix}
    g(u, x) \\
    h\frac{\partial{g(u,x)}}{\partial{u}} \\
    h\frac{\partial{g(u,x)}}{\partial{x}}
  \end{pmatrix}
  +\frac{1}{2}\mathbf{Q}_g(u,x)+o_p(h^2\mathbf{1}_T)\\
  &= \mathbf{Z}^*(u,x)\begin{pmatrix}
    g(u, x) \\
    h\frac{\partial{g(u,x)}}{\partial{u}} \\
    h\frac{\partial{g(u,x)}}{\partial{x}}
  \end{pmatrix}+\mathbf{R}^*(u,x)\begin{pmatrix}
    \frac{h^2}{2}\frac{\partial^2{g(u,x)}}{\partial^2{u}} \\
    \frac{h^2}{2}\frac{\partial^2{g(u,x)}}{\partial^2{x}} \\
    \frac{h^2}{2}\frac{\partial^2{g(u,x)}}{\partial{u}\partial{x}} \\
  \end{pmatrix}
  +o_p(h^2\mathbf{1}_T).
\end{align*}
Thus the bias part $B(u,x)$ can be written as
\begin{align}\label{eqn:ghat-asymptotic-bias-expansion}
  B(u,x) &= \big(\mathbf{Z}^*(u,x)^\top\mathbf{W}^*(u,x)\mathbf{Z}^*(u,x)\big)^{-1}\mathbf{Z}^*(u,x)^\top\mathbf{W}^*(u,x)\mathbf{R}^*(u,x)\begin{pmatrix}
    \frac{h^2}{2}\frac{\partial^2{g(u,x)}}{\partial^2{u}} \\
    \frac{h^2}{2}\frac{\partial^2{g(u,x)}}{\partial^2{x}} \\
    \frac{h^2}{2}\frac{\partial^2{g(u,x)}}{\partial{u}\partial{x}} \\
  \end{pmatrix}+o_p(h^2).
\end{align}
Note that
\begin{equation*}
  \frac{1}{T}\mathbf{Z}^*(u,x)^\top\mathbf{W}^*(u,x)\mathbf{Z}^*(u,x)=\begin{pmatrix}
      M_{0,0} & M_{1,0} & M_{0, 1}\\
      M_{1,0} & M_{2,0} & M_{1, 1}\\
      M_{0,1} & M_{1,1} & M_{0, 2}\\
     \end{pmatrix}, \\
\end{equation*}
where
\begin{align*}
M_{i,j} &= M_{i,j}(u, x)= \frac{1}{T}\sum_{t=1}^TK_h(t/T-u)K_h(X_t-x)\big(\frac{t/T-u}{h}\big)^{i}(\frac{X_t(t/T)-x}{h})^{j}, ~~~i, j = 0, 1, 2.
\end{align*}
With the similar arguments as in \cite{ruppert1994multivariate}, we have
\begin{equation}\label{eqn:ztwz-star-limit}
  \frac{1}{T}\mathbf{Z}^*(u,x)^\top\mathbf{W}^*(u,x)\mathbf{Z}^*(u,x)=f(u, x)\otimes\diag\{1,\mu_2,\mu_2\}(1+o_p(h^2)),
\end{equation}
and
\begin{equation}\label{eqn:ztwr-star-limit}
  \frac{1}{T}\mathbf{Z}^*(u,x)^\top\mathbf{W}^*(u,x)\mathbf{R}^*(u,x)=f(u, x)\otimes
  \begin{pmatrix}
  \mu_2 & \mu_2 & 0 \\
  0 & 0 & 0 \\
  0 & 0 & 0 \\
  \end{pmatrix}
  (1+o_p(h^3)).
\end{equation}
It follows from \eqref{eqn:ghat-asymptotic-bias-expansion}, \eqref{eqn:ztwz-star-limit}, and \eqref{eqn:ztwr-star-limit} that the asymptotic bias
\begin{align}
  B(u,x)&=
  \begin{pmatrix}
      1 & 0 & 0 \\
      0 & \mu_2^{-1} & 0 \\
      0 & 0 & \mu_2{-1} \\
  \end{pmatrix}(1+o_p(1))
  \times
  \begin{pmatrix}
      \mu_2 & \mu_2 & 0 \\
      0 & 0 & 0 \\
      0 & 0 & 0 \\
  \end{pmatrix}(1+o_p(h^3))\times\begin{pmatrix}
    \frac{h^2}{2}\frac{\partial^2{g(u,x)}}{\partial^2{u}} \\
    \frac{h^2}{2}\frac{\partial^2{g(u,x)}}{\partial^2{x}} \\
    \frac{h^2}{2}\frac{\partial^2{g(u,x)}}{\partial{u}\partial{x}} \\
  \end{pmatrix}+o_p(h^2) \notag\\
  &=
  \begin{pmatrix}
    \frac{\mu_2h^2}{2}\Big(\frac{\partial^2{g(u,x)}}{\partial^2{u}}+\frac{\partial^2{g(u,x)}}{\partial^2{x}}\Big) \\
    0 \\
    0
  \end{pmatrix}+o_p(h^2). \label{eqn:ghat-asymptotic-bias}
\end{align}

Next we consider the variance part. The following proof is completed by a classical blocking technique \cite{cai2000functional}. Partition $\{1,2,\cdots,T\}$ into $2q_T+1$ subsets with small-block of size $s=s_T$ and large-block of size $l=l_T$. Let $q=\lfloor\frac{T}{r+s}\rfloor$, $\mathbf{N}_T=\frac{1}{T}\mathbf{Z}^*(u,x)^\top\mathbf{W}^*(u,x)\mathbf{e}$. For any non-zero vector $\boldsymbol{\xi}=(\xi_1,\xi_2,\xi_3)^\top$, let
\begin{align*}
  Q_T &= \boldsymbol{\xi}^\top\mathbf{N}_T=\frac{1}{T}\sum_{t=1}^T\Big\{\xi_1+\xi_2\big(\frac{t/T-u}{h}\big)+\xi_3\big(\frac{X_t(t/T)-x}{h}\big)\Big\}K_h(t/T-u)K_h(X_t(t/T)-x)e_t \\
    &\stackrel{\vartriangle}{=}\frac{1}{T}\sum_{t=1}^TZ_t,
\end{align*}
where $Z_t=(\xi_1,\xi_2,\xi_3)\begin{pmatrix} 1 \\ \frac{t/T-u}{h} \\ \frac{X_t(t/T)-x}{h} \end{pmatrix}K_h(t/T-u)K_h(X_t(t/T)-x)e_t$. Thus
\begin{align}
  Var(\sqrt{Th^2}Q_t) = Var(\frac{h}{\sqrt{T}}\sum_{t=1}^TZ_t)=\frac{h^2}{T}\sum_{t=1}^TVar(Z_t)+\frac{2h^2}{T}\sum_{l=1}^{T-1}\sum_{t=1}^{T-l}cov(Z_t, Z_{t+l}). \label{eqn:Qt-var}
\end{align}
Let $w_t=K_h^2(t/T-u)K_h^2(X_t(t/T)-x)$. It's easy to show that $E(Z_t)=0$ and
\begin{align}
  &Var(Z_t) \\=& E\Big[\Big\{\xi_1+\xi_2\big(\frac{t/T-u}{h}\big)+\xi_3\big(\frac{X_t(t/T)-x}{h}\big)\Big\}K_h(t/T-u)K_h(X_t(t/T)-x)e_t\Big]^2 \notag \\
    =& E\Big[\Big\{\xi_1+\xi_2\big(\frac{t/T-u}{h}\big)+\xi_3\big(\frac{X_t(t/T)-x}{h}\big)\Big\}K_h(t/T-u)K_h(X_t(t/T)-x)\Big]^2E(e_t^2) \notag \\
    =& E(e_t^2)\boldsymbol{\xi}^\top
      \begin{pmatrix}
        w_t & \big(\frac{t/T-u}{h}\big)w_t & \big(\frac{X_t(t/T)-x}{h}\big)w_t \\
        \big(\frac{t/T-u}{h}\big)w_t & \big(\frac{t/T-u}{h}\big)^2w_t & \big(\frac{t/T-u}{h}\big)\big(\frac{X_t(t/T)-x}{h}\big)w_t \\
        \big(\frac{X_t(t/T)-x}{h}\big)w_t & \big(\frac{t/T-u}{h}\big)\big(\frac{X_t(t/T)-x}{h}\big)w_t & \big(\frac{X_t(t/T)-x}{h}\big)^2w_t \\
      \end{pmatrix}
     \boldsymbol{\xi}\{1+O_p(h)\} \notag \\
    =& \frac{E(e_t^2)}{h^2}f(u, x)\boldsymbol{\xi}^\top
      \begin{pmatrix}
        \nu_0^2 & \nu_0\nu_1 & \nu_0\nu_1 \\
        \nu_0\nu_1 & \nu_0\nu_2 & \nu_1^2 \\
        \nu_0\nu_1 & \nu_1^2 & \nu_0\nu_2 \\
      \end{pmatrix}
     \boldsymbol{\xi}\{1+O_p(h)\}. \label{eqn:zt-var}
\end{align}
As shown in \cite{kim2001nonparametric},
\begin{equation}\label{eqn:et-var}
  E(e_t^2)=\gamma_0(t/T)+o(1),
\end{equation}
Hence, it follows from that \eqref{eqn:zt-var} and \eqref{eqn:et-var} that
\begin{equation*}
  \frac{1}{T}\sum_{t=1}^TVar(Z_t)=\frac{1}{h^2}\int_0^1\gamma_0(u)du f(u, x)\boldsymbol{\xi}^\top
      \begin{pmatrix}
        \nu_0^2 & \nu_0\nu_1 & \nu_0\nu_1 \\
        \nu_0\nu_1 & \nu_0\nu_2 & \nu_1^2 \\
        \nu_0\nu_1 & \nu_1^2 & \nu_0\nu_2 \\
      \end{pmatrix}
     \boldsymbol{\xi}\{1+O_p(h)\}.
\end{equation*}
Since $\{e_t\}$ is $\alpha$-mixing, it follows from Lemma (1) of \cite{cai2000functional} that
$ h\sum_{t=1}^{T-1}|cov(Z_1,Z_{t+1})|=o(1)$, which implies that the second term of \eqref{eqn:Qt-var} is negligible. Hence,
\begin{align*}
  Var(\sqrt{Th^2}Q_t)&=\int_0^1\gamma_0(u)du f(u, x)\boldsymbol{\xi}^\top
      \begin{pmatrix}
        \nu_0^2 & \nu_0\nu_1 & \nu_0\nu_1 \\
        \nu_0\nu_1 & \nu_0\nu_2 & \nu_1^2 \\
        \nu_0\nu_1 & \nu_1^2 & \nu_0\nu_2 \\
      \end{pmatrix}
     \boldsymbol{\xi}\{1+O_p(h)\}.
\end{align*}
For $0\leq j\leq q-1$, define
\begin{equation*}
  \eta_j=\sum_{t=j(l+s)+1}^{j(l+s)+l}Z_t;\text{~~~~} \xi_j=\sum_{t=j(l+s)+l+1}^{(j+1)(l+s)}Z_t;\text{~~~~} \zeta_q=\sum_{t=q(l+s)+1}^{T}Z_t.
\end{equation*}
Then
\begin{equation*}
  \sqrt{Th^2}Q_t=\frac{h}{\sqrt{T}}(\sum_{j=0}^{q-1}\eta_t+\sum_{j=0}^{q-1}\xi_t+\zeta_q)=\frac{h}{\sqrt{T}}(Q_1+Q_2+Q_3).
\end{equation*}
With the similar arguments of Theorem 2 in \cite{cai2000functional}, it can be shown that small block $Q_2$ and the remainder $Q_3$ are asymptotically negligible in probability, and each $\pi_j$ in large block $Q_1$ is asymptotically independent under condition (C2). Thus the asymptotic normality of $Q_1$ is derived by Lindeberg Theorem, and it holds that
\begin{equation*}
  \sqrt{Th^2}\mathbf{N}_T~~\stackrel{\mathcal{D}}{\rightarrow}~~N\Bigg(\mathbf{0},\int_0^1\gamma_0(u)du f(u, x)
      \begin{pmatrix}
        \nu_0^2 & \nu_0\nu_1 & \nu_0\nu_1 \\
        \nu_0\nu_1 & \nu_0\nu_2 & \nu_1^2 \\
        \nu_0\nu_1 & \nu_1^2 & \nu_0\nu_2 \\
      \end{pmatrix}\Bigg), ~~~~~~ \mathrm{as}~ T\to\infty.
\end{equation*}
 It together with \eqref{eqn:ztwz-star-limit} follows that
\begin{align}\label{eqn:ghat-asymptotic-variance}
\sqrt{Th^2}\big(\mathbf{Z}^*(u,x)^\top\mathbf{W}^*(u,x)\mathbf{Z}^*(u,x)\big)^{-1}\mathbf{Z}^*(u,x)^\top\mathbf{W}^*(u,x)\mathbf{e}
\stackrel{\mathcal{D}}{\rightarrow} N(\mathbf{0},\mathbf{V}_{u,x}), ~~~ \mathrm{as}~ T\to\infty,
\end{align}
where
\begin{equation*}
\mathbf{V}_{u,x}=\frac{\int_0^1\gamma_0(u)du}{f(u, x)}
      \begin{pmatrix}
        \nu_0^2 & \frac{\nu_0\nu_1}{\mu_2} & \frac{\nu_0\nu_1}{\mu_2} \\
        \frac{\nu_0\nu_1}{\mu_2} & \frac{\nu_0\nu_2}{\mu_2^2} & \frac{\nu_1^2}{\mu_2^2} \\
        \frac{\nu_0\nu_1}{\mu_2} & \frac{\nu_1^2}{\mu_2^2} & \frac{\nu_0\nu_2}{\mu_2^2} \\
      \end{pmatrix}.
\end{equation*}
Combining the results of \eqref{eqn:ghat-asymptotic-bias} and \eqref{eqn:ghat-asymptotic-variance}, the theorem is proved.
\end{proof}

Theorem \ref{thm:trend-convergence-rate} and \ref{thm:trend-asymptotic-normality} show that the preliminary estimator $\hat{g}(u,x)$ is consistent and asymptotic normal. However, it's not efficient due to the fact that the error structure is not taken into account for estimation. On the other side, \cite{liyouJASM} use it to estimate the error structure and then refine the estimator of $g(u,x)$ with the fitted error structure.

\subsection{Asymptotic results on the estimator of the error term}
In this section, we discuss the asymptotic property for the estimator of the error term,  $\hat{\boldsymbol\phi}(u)$, which is given by \cite{liyouJASM}. To do so, we need some additional conditions except for conditions (C1)-(C7).
\begin{enumerate}[(C1)]
\setcounter{enumi}{7}
  \item The function $\phi_i(\cdot)$ $(i=1,2,\cdots,p)$ is twice continuously differentiable in $u$ with uniformly bounded second ordered derivative, and the root of $\Phi(u, z)=1-\sum_{i=1}^p\phi_i(u)z^i$ are bounded away from the unit circle for each $u\in[0,1]$.
  \item As $T\to\infty$, $h_e=O(T^{-\frac{1}{5}}), h=o(T^{-\frac{1}{5}}), T^{-\frac{2}{5}}h^{-2}\log{T}\to0$.
  \item $E(\epsilon_t^4)<\infty$.
  \item $\sup_t\sum_{j=0}^\infty j^{\frac{1}{2}}\varphi_j(t/T)^2<\infty, \sup_t\sum_{j=0}^\infty j^{\frac{1}{2}}(\varphi_j'(t/T))^2<\infty.$
\end{enumerate}
\begin{remark} Condition (C8) guarantees that the approximate stationary time series $\{e_t(u)\}$ of $\{e_t\}$ around the neighborhood of $u$ is causal and satisfies $\sum_{k=-\infty}^{\infty}\vert\gamma_k(u)\vert<\infty$, where $\gamma_k(u)$, defined in \eqref{eqn:Gamma}, is the $k$-th ordered auto-covariance function of $\{e_t(u)\}$. (C9) is the condition to obtain the optimal convergence for the estimator of $\phi_k(\cdot)$. Conditions (C10) and (C11) are the same as the ones in \cite{kim2001nonparametric} for the asymptotic normality of time-varying autoregressive coefficient functions.
\end{remark}

Theorem \ref{thm:alpha-asymptotic-normality} presents that the asymptotic normality of $\hat{\boldsymbol\alpha}(u)$. To prove it, we first show a proposition for the asymptotic normality of $\tilde{\boldsymbol\alpha}(u)$, where $\tilde{\boldsymbol\alpha}(u)=\tilde{\mathbf{S}}_T^{-1}\tilde{\mathbf{t}}_T$ is the estimate of $\boldsymbol\alpha(u)$ with $\hat{e}_t$ in \eqref{solution:tvAR} replaced by its true value $e_t$. Note that this proposition is the same as the i.i.d case in \cite{cai2000efficient}. Its proof is similar to Theorem 5 of \cite{kim2001nonparametric} and therefore omitted here.

\begin{proposition}\label{prop:phi-asymptotics}
Under conditions (C1)-(C11), for any $u\in(0, 1)$,
\begin{equation*}
\sqrt{Th_e}\Big(\tilde{\boldsymbol{\alpha}}(u)-\boldsymbol{\alpha}(u)
  -\frac{h_e^2}{2}\mu_2
    \begin{pmatrix}
        \boldsymbol\phi''(u) \\
        \mathbf{0}
    \end{pmatrix}
+o(h_e^2)\Big)~\stackrel{D}{\longrightarrow}~\mathcal{N}(\mathbf{0},\boldsymbol{\Sigma}(u)),~~~~~ \mathrm{as}~T\to\infty,
\end{equation*}
where $\boldsymbol{\Sigma}(u)$ is defined in \eqref{eqn:Sigma(u)}.
\end{proposition}

\begin{theorem}\label{thm:alpha-asymptotic-normality}
Assume that conditions (C1)-(C11) are satisfied, $r=\min\{\rho, 1\}>\frac{1}{2}$, for any $u\in(0, 1)$, it holds that
\begin{equation*}
\sqrt{Th_e}\Big(\hat{\boldsymbol{\alpha}}(u)-\boldsymbol{\alpha}(u)
  -\frac{h_e^2}{2}\mu_2
    \begin{pmatrix}
        \boldsymbol\phi''(u) \\
        \mathbf{0}
    \end{pmatrix}
+o(h_e^2)\Big)~\stackrel{D}{\longrightarrow}~\mathcal{N}(\mathbf{0},\boldsymbol{\Sigma}(u)),~~~~~ \mathrm{as}~T\to\infty,
\end{equation*}
where $\boldsymbol\phi^{''}(u)=(\phi_1^{''}(u),\cdots,\phi_p^{''}(u))$, $\boldsymbol\Gamma(u)$ is given by \eqref{eqn:Gamma}, and
\begin{equation}\label{eqn:Sigma(u)}
  \boldsymbol\Sigma(u) = \begin{pmatrix}
     \nu_0 & \mu_1^{-1}\nu_1 \\
    \mu_1^{-1}\nu_1 & \mu_1^{-2}\nu_2
  \end{pmatrix}\otimes\boldsymbol\Gamma^{-1}(u).
\end{equation}
\end{theorem}

\begin{proof}
With Slutsky Theorem, the asymptotic normality can be obtained immediately from \eqref{eqn:phi-convegence-in-probability} and Proposition \ref{prop:phi-asymptotics}. Therefore we just prove that
\begin{align}\label{eqn:phi-convegence-in-probability}
\hat{\boldsymbol\alpha}(u)-\tilde{\boldsymbol\alpha}(u)=o_p((Th_e)^{-\frac{1}{2}}).
\end{align}

Define a $2p\times2p$ matrix $\hat{\mathbf{S}}_T$ and a $2p$-dimensional vector $\hat{\mathbf{t}}_T$ as follows,
\begin{align}
  \hat{\mathbf{S}}_T &= \frac{1}{T}\hat{\mathbf{Z}}_u^\top\mathbf{W}_e\hat{\mathbf{Z}}_u=
     \begin{pmatrix}
      \hat{\mathbf{S}}_{T,0}(u) & \hat{\mathbf{S}}_{T,1}(u)\\
      \hat{\mathbf{S}}_{T,1}(u) & \hat{\mathbf{S}}_{T,2}(u)
     \end{pmatrix}, \label{eqn:S-t}\\
  \hat{\mathbf{t}}_T &= \frac{1}{T}\hat{\mathbf{Z}}_u^\top\mathbf{W}_e\hat{\mathbf{e}}_T=(\hat{\mathbf{t}}_{T,0}(u),\hat{\mathbf{t}}_{T,1}(u))^\top, \label{eqn:t-t}
\end{align}
where
\begin{align}
\hat{\mathbf{S}}_{T,i}(u)&=\frac{1}{T}\sum_{t=p+1}^TK_{h_e}(t/T-u)(\frac{t/T-u}{h_e})^{i}\hat{\mathbf{e}}_{t-1}\hat{\mathbf{e}}_{t-1}^\top, ~~~i = 0, 1, 2, 3, \label{eqn:Sij}\\
\hat{\mathbf{t}}_{T, i}(u)&=\frac{1}{T}\sum_{t=p+1}^TK_{h_e}(t/T-u)(\frac{t/T-u}{h_e})^{i}\hat{\mathbf{e}}_{t-1}\hat{e}_{t},~~~~ i=0, 1. \label{eqn:tij}
\end{align}
So \eqref{solution:tvAR} can be represented as $\hat{\boldsymbol\alpha}(u)=
\hat{\mathbf{S}}_T^{-1}\hat{\mathbf{t}}_T$.

By \eqref{eqn:Sij}, the $(r, s)$-th element of $\hat{\mathbf{S}}_{T,i}(u)$ is
\begin{align*}
  \hat{S}_{T,i}(u, r, s) & = \frac{1}{T}\sum_{t=p+1}^TK_{h_e}(t/T-u)\hat{e}_{t-r}\hat{e}_{t-s}\Big(\frac{t/T-u}{h_e}\Big)^{i}\\
   & = \frac{1}{T}\sum_{t=p+1}^TK_{h_e}(t/T-u)\Big[Y_{t-r}-\hat{g}\Big(\frac{t-r}{T}\Big)\Big]\Big[Y_{t-s}-\hat{g}\Big(\frac{t-s}{T}\Big)\Big]\Big(\frac{t/T-u}{h_e}\Big)^{i} \\
   & = \frac{1}{T}\sum_{t=p+1}^TK_{h_e}(t/T-u)\Big[e_{t-r}+g\Big(\frac{t-r}{T}\Big)-\hat{g}\Big(\frac{t-r}{T}\Big)\Big]
       \Big[e_{t-s}+g\Big(\frac{t-s}{T}\Big)-\hat{g}\Big(\frac{t-s}{T}\Big)\Big]\\
   & ~~~~~\times\Big(\frac{t/T-u}{h_e}\Big)^{i} \\
   & = \tilde{S}_{T,i}(u, r, s)+\hat{s}_1(u, r, s)+\hat{s}_2(u, r, s)+\hat{s}_3(u, r, s),
\end{align*}
where $\tilde{S}_{T,i}(u, r, s)=\frac{1}{T}\sum_{t=p+1}^TK_{h_e}(t/T-u)e_{t-r}e_{t-s}\Big(\frac{t/T-u}{h_e}\Big)^{i}$ is the $(r, s)$-th element of $\tilde{\mathbf{S}}_{T,i}(u)$ and
\begin{align*}
  \hat{s}_1(u, r, s) & = \frac{1}{T}\sum_{t=p+1}^TK_{h_e}(t/T-u)\Big[g\Big(\frac{t-r}{T}\Big)-\hat{g}\Big(\frac{t-r}{T}\Big)\Big]
                         \Big[g\Big(\frac{t-s}{T}\Big)-\hat{g}\Big(\frac{t-s}{T}\Big)\Big]\Big(\frac{t/T-u}{h_e}\Big)^{i}, \\
  \hat{s}_2(u, r, s) & = \frac{1}{T}\sum_{t=p+1}^TK_{h_e}(t/T-u)e_{t-r}\Big[g\Big(\frac{t-s}{T}\Big)-\hat{g}\Big(\frac{t-s}{T}\Big)\Big]\Big(\frac{t/T-u}{h_e}\Big)^{i}, \\
  \hat{s}_3(u, r, s) & = \frac{1}{T}\sum_{t=p+1}^TK_{h_e}(t/T-u)e_{t-s}\Big[g\Big(\frac{t-r}{T}\Big)-\hat{g}\Big(\frac{t-r}{T}\Big)\Big]\Big(\frac{t/T-u}{h_e}\Big)^{i}.
\end{align*}
Similarly, the $r$-th element of $\hat{\mathbf{t}}_{T, i}(u)$ is $
\hat{t}_{T, i}(u, r) = \tilde{t}_{T, i}(u, r)+\hat{t}_1(u,r)+\hat{t}_2(u,r)+\hat{t}_3(u,r)$,
where $\tilde{t}_{T, i}(u, r) = \frac{1}{T}\sum_{t=p+1}^TK_{h_e}(t/T-u)e_{t}e_{t-r}\Big(\frac{t/T-u}{h_e}\Big)^{i}$ is the $r$-th element of $\hat{\mathbf{t}}_{T, i}(u)$ and
\begin{align*}
  \hat{t}_1(u,r)  & = \frac{1}{T}\sum_{t=p+1}^TK_{h_e}(t/T-u)\Big[g\Big(\frac{t}{T}\Big)-\hat{g}\Big(\frac{t}{T}\Big)\Big]
                        \Big[g\Big(\frac{t-r}{T}\Big)-\hat{g}\Big(\frac{t-r}{T}\Big)\Big]\Big(\frac{t/T-u}{h_e}\Big)^{i}, \\
  \hat{t}_2(u,r)  & = \frac{1}{T}\sum_{t=p+1}^TK_{h_e}(t/T-u)e_{t}\Big[g\Big(\frac{t-r}{T}\Big)-\hat{g}\Big(\frac{t-r}{T}\Big)\Big]\Big(\frac{t/T-u}{h_e}\Big)^{i}, \\
  \hat{t}_3(u,r)  & = \frac{1}{T}\sum_{t=p+1}^TK_{h_e}(t/T-u)e_{t-r}\Big[g\Big(\frac{t}{T}\Big)-\hat{g}\Big(\frac{t}{T}\Big)\Big]\Big(\frac{t/T-u}{h_e}\Big)^{i}.
\end{align*}

With the decomposition of $\hat{S}_{T,i}(u, r, s)$ and $\hat{t}_{T,i}(u, r)$, \eqref{eqn:phi-convegence-in-probability} is immediately obtained from the fact that $\hat{\boldsymbol\alpha}(u)-\tilde{\boldsymbol\alpha}(u)=\hat{\mathbf{S}}_T^{-1}(\hat{\mathbf{t}}_T-\tilde{\mathbf{t}}_T)+
 \hat{\mathbf{S}}_T^{-1}(\tilde{\mathbf{S}}_T-\hat{\mathbf{S}}_T)\tilde{\mathbf{S}}_T^{-1}\tilde{\mathbf{t}}_T$ and the two statements below:
\begin{enumerate}[(i)]
  \item $\hat{s}_1(u, r, s)+\hat{s}_2(u, r, s)+\hat{s}_3(u, r, s)=o_p((Th_e)^{-\frac{1}{2}}),~~~~ i=0, 1, 2;$
  \item $\hat{t}_1(u, r)+\hat{t}_2(u, r)+\hat{t}_3(u, r)=o_p((Th_e)^{-\frac{1}{2}}), ~~~~i=0,1.$
\end{enumerate}
Since the proof for (i) is similar as (ii), here we only demonstrate (ii).

Note that $\frac{t/T-u}{h_e}\leq C_1$ on the compact support of $K_{h_e}(t/T-u)$, thus we only present the proof for the case $i=0$.

Firstly,
\begin{align*}
  \vert\hat{t}_1(u, r)\vert &= \frac{1}{T}\sum_{t=p+1}^TK_{h_e}(t/T-u)\Big\vert g\Big(\frac{t}{T}\Big)-\hat{g}\Big(\frac{t}{T}\Big)\Big\vert\cdot
                        \Big\vert g\Big(\frac{t-r}{T}\Big)-\hat{g}\Big(\frac{t-r}{T}\Big)\Big\vert \\
                      &= \frac{1}{Th}\sum_{t=p+1}^TK\Big(\frac{t/T-u}{h_e}\Big)\Big\vert g\Big(\frac{t}{T}\Big)-\hat{g}\Big(\frac{t}{T}\Big)\Big\vert\cdot
                        \Big\vert g\Big(\frac{t-r}{T}\Big)-\hat{g}\Big(\frac{t-r}{T}\Big)\Big\vert \\
                      &\leq \frac{C}{Th}\sum_{t=p+1}^T\Big\vert g\Big(\frac{t}{T}\Big)-\hat{g}\Big(\frac{t}{T}\Big)\Big\vert\cdot
                        \Big\vert g\Big(\frac{t-r}{T}\Big)-\hat{g}\Big(\frac{t-r}{T}\Big)\Big\vert.
\end{align*}
By Theorem \ref{thm:trend-convergence-rate} and condition (C9),
\begin{align*}
  \vert\hat{t}_1(u, r)\vert &\leq \frac{C}{h_e}\cdot\Big(\max_{1\leq t\leq T}\Big\vert g\Big(\frac{t}{T}\Big)-\hat{g}\Big(\frac{t}{T}\Big)\Big\vert\Big)^2 = \frac{C}{h_e}\cdot O_p\Big(\frac{\log T}{Th^2}+\frac{1}{T^{2r}h^2}+h^4\Big) = o_p((Th_e)^{-\frac{1}{2}}).
\end{align*}

Next we prove that $\hat{t}_3(u, r)=o_p((Th_e)^{-\frac{1}{2}})$.
\begin{align*}
\hat{t}_3(u, r) &= \frac{1}{T}\sum_{t=p+1}^TK_{h_e}(t/T-u)(e_{t-r}-e_{t-r}(u))\Big[g\Big(\frac{t}{T}\Big)-\hat{g}\Big(\frac{t}{T}\Big)\Big]\\
  &~~~~~+ \frac{1}{T}\sum_{t=p+1}^TK_{h_e}(t/T-u)e_{t-r}(u)\Big[g\Big(\frac{t}{T}\Big)-\hat{g}\Big(\frac{t}{T}\Big)\Big] \\
  &\stackrel{\vartriangle}{=}M_1+M_2.
\end{align*}
By (C1), it's easy to show that
\begin{align*}
M_1 &= \frac{1}{Th_e}\sum_{t=p+1}^TK(\frac{t/T-u}{h_e})(e_{t-r}-e_{t-r}(u))\Big[g\Big(\frac{t}{T}\Big)-\hat{g}\Big(\frac{t}{T}\Big)\Big] \\
    &\leq \frac{1}{Th_e}\sum_{|t/T-u|\leq C_1{h_e}} C\cdot O_p(\Big\vert\frac{t}{T}-u\Big\vert+\frac{1}{T}\Big)\cdot O_p\Big(\sqrt{\frac{\log{T}}{Th^2}}+\frac{1}{T^rh}+h^2\Big) \\
    &\leq \frac{(2Th_e+1)C}{Th_e}\cdot O_p(h_e)\cdot O_p\Big(\sqrt{\frac{\log{T}}{Th^2}}+h^2\Big) \\
    &= O_p\Big(h_e\sqrt{\frac{\log{T}}{Th^2}}+h_eh^2\Big) =o_p((Th_e)^{-\frac{1}{2}}).
\end{align*}
On the other hand, notice that $E(M_2)=0$, we only need to show $E(Th_e\cdot M_2^2)\to0$ for $M_2=o_p((Th_e)^{-\frac{1}{2}})$. We have
\begin{align*}
  E(Th_e\cdot M_2^2) &= \frac{1}{Th_e}E\Big\{\sum_{t=p+1}^TK(\frac{t/T-u}{h_e})e_{t-r}(u)\Big[g\Big(\frac{t}{T}\Big)-\hat{g}\Big(\frac{t}{T}\Big)\Big]\Big\}^2 \\
   &\leq \frac{C}{Th_e}\cdot (\sup_{x\in[0, 1]}\vert g(x)-\hat{g}(x)\vert)^2\cdot E\Big(\sum_{|t/T-u|\leq h}e_{t-r}(u)\Big)^2 \\
   &\leq \frac{(2Th_e+1)C}{Th_e}\cdot O_p\Big(\frac{\log{T}}{Th^2}+h^4\Big)\cdot \sum_{k=0}^\infty\vert\gamma_k(u)\vert ~~\longrightarrow 0,
\end{align*}
Thus $\hat{t}_{r3}(u)=o_p((Th_e)^{-\frac{1}{2}}).$ Similarly it holds that $\hat{t}_{r2}(u)=o_p((Th_e)^{-\frac{1}{2}})$. Then argument (ii) holds. The proof is completed.
\end{proof}

Proposition \ref{prop:phi-asymptotics} and Theorem \ref{thm:alpha-asymptotic-normality} show that the estimator $\hat{\boldsymbol{\phi}}(u)$  has the same convergence rate and the asymptotic distribution as the estimator when $e_t$ is observable as shown in \cite{kim2001nonparametric}.  Moreover, for the case that the order $p$ of error series is unknown, \cite{liyouJASM} proposed the estimate of $\mathcal{S}_1$ and $\mathcal{S}_2$, viz., $\hat{\mathcal{S}}_{1\hat{\lambda}}$ and $\hat{\mathcal{S}}_{2\hat{\gamma}}$ respectively, which are used to identify the coefficient functions and the constant coefficients. Theorem \ref{thm:BIC-consistency} shows that $\hat{\mathcal{S}}_{1\hat{\lambda}}$ and $\hat{\mathcal{S}}_{2\hat{\gamma}}$ can identify the true model consistently. To prove it, we first present some lemmas.

\begin{lemma}\label{lemma:uniform-consistency-x-unavailable}
Assume that conditions (C1)-(C11) are satisfied. $\Gamma(u)$ defined in \eqref{eqn:Gamma} is nonsingular for all $u\in(0, 1)$ and has uniformly bounded second derivatives. When $(h_e\lambda)/\sqrt{Th_e\log T}$ $\to 0$, and $\gamma/\sqrt{Th_e\log(T)}\to 0$ as $T\to\infty$, we have
\begin{align*}
  \sup_u\Vert\hat{\boldsymbol{\phi}}_\lambda(u)-\boldsymbol{\phi}^*(u)\Vert&=O_p(c_T),  \\
  \sup_u\Vert h_e(\hat{\boldsymbol{\phi}}^{'}_\gamma(u)-\boldsymbol{\phi}^{*'}(u))\Vert&=O_p(c_T),
\end{align*}
where $c_T=\big(\log(Th_e)/Th_e\big)^{1/2}$ and $\Vert\cdot\Vert$ is the $L_2$ norm.
\end{lemma}

\begin{proof}
The result of Lemma \ref{lemma:uniform-consistency-x-unavailable} can be directly derived from
\begin{equation}\label{eqn:ULASSO-uniform-consistency}
    \sup_u\Vert\hat{\boldsymbol{\alpha}}_{\lambda, \gamma}(u)-\boldsymbol{\alpha}^*(u)\Vert=O_p(c_T).
\end{equation}
To prove \eqref{eqn:ULASSO-uniform-consistency}, we firstly define a ball $B_C=\{\boldsymbol{\alpha}(u): \boldsymbol{\alpha}(u)= \boldsymbol{\alpha}^*(u)+c_T\mathbf{r},\Vert \mathbf{r}\Vert\leq C\}$. By \cite{fan2001variable}, we only need to show that for any $\epsilon>0$, there exists $C>0$ which doesn't depend on $u$ such that
\begin{equation}\label{eqn:fan&li-sufficient-condition}
  P\Big(\inf_{\Vert r\Vert=C}Q(\boldsymbol{\alpha}^*(u)+c_T\mathbf{r}\big|\hat{e}, u)>Q(\boldsymbol{\alpha}^*(u)\big|\hat{e}, u)\Big)\geq 1-\epsilon.
\end{equation}

Let $\Vert \mathbf{r}\Vert=C$. For any $u\in(0, 1)$, by definition of $Q(\boldsymbol{\alpha}(u)\big|\hat{X}, u)$, we have
\begin{align*}
  R_1 &= \frac{h_e}{\log(1/h_e)}\Big(Q(\boldsymbol{\alpha}^*(u)+c_T\mathbf{r}\big|\hat{e}, u)-Q(\boldsymbol{\alpha}^*(u)\big|\hat{e}, u)\Big) \\
      &= \frac{h_e}{\log(1/h_e)}\Big\{\sum_{i=1}^T\Big(\hat{\mathbf{e}}_T-\hat{\mathbf{Z}}_u(\boldsymbol{\alpha}^*(u)+c_T\mathbf{r})\Big)^\top
        \mathbf{W}_e\Big(\hat{\mathbf{e}}_T-\hat{\mathbf{Z}}_u(\boldsymbol{\alpha}^*(u)+c_T\mathbf{r})\Big) \\
      &~~~~-\sum_{i=1}^T\Big(\hat{\mathbf{e}}_T-\hat{\mathbf{Z}}_u\boldsymbol{\alpha}^*(u)\Big)^\top
        \mathbf{W}_e\Big(\hat{\mathbf{e}}_T-\hat{\mathbf{Z}}_u\boldsymbol{\alpha}^*(u)\Big)\Big\} \\
      &~~~~+\sum_{k=1}^p\frac{h_e\lambda}{\log(1/h_e)w_k}\Big(\Big|\phi^*_k(u)+c_Tr_k\Big|-|\phi^*_k(u)|\Big) \\
      &~~~~+\sum_{k=1}^p\frac{\gamma}{\log(1/h_e)w'_k}\Big(\Big|h_e\phi^{*'}_k(u)+c_Tr_{p+k}\Big|-|\phi^{*'}_k(u)|\Big).
\end{align*}
By some simple calculation, we have
\begin{align*}
  R_1 &\geq \frac{\mathbf{r}^\top\hat{\mathbf{Z}}_u^\top\mathbf{W}_e\hat{\mathbf{Z}}_u\mathbf{r}}{T}
        -2\frac{\mathbf{r}^\top}{\log(1/h_e)}\sqrt{\frac{h_e}{T}}\hat{\mathbf{Z}}_u^\top\mathbf{W}_e(\hat{\mathbf{e}}_T-\hat{\mathbf{Z}}_u\boldsymbol{\alpha}^*(u)) \\
      &~~~~+\sum_{k\in\mathcal{S}_1}\frac{h_e\lambda}{\log(1/h_e)w_k}\Big(\Big|\phi^*_k(u)+c_Tr_k\Big|-|\phi^*_k(u)|\Big) \\
      &~~~~+\sum_{k\in\mathcal{S}_2}\frac{\gamma}{\log(1/h_e)w'_k}\Big(\Big|h_e\phi^{*'}_k(u)+c_Tr_{p+k}\Big|-|\phi^{*'}_k(u)|\Big).
\end{align*}
Note that by Theorem \ref{thm:trend-convergence-rate}, it holds that $\sup_{t}\vert \hat{e}_t-e_t\vert=\sup_{t}\vert\hat{g}(t/T,X_t)-g(t/T,X_t)\vert=O_p\Big(h_e\sqrt{\frac{\log{T}}{Th^2}}+h_eh^2\Big)$, which follows that
\begin{align}
  \frac{\mathbf{r}^\top\hat{\mathbf{Z}}_u^\top\mathbf{W}_e\hat{\mathbf{Z}}_u\mathbf{r}}{T}&=\frac{\mathbf{r}^\top{\mathbf{Z}}_u^\top\mathbf{W}_e{\mathbf{Z}}_u\mathbf{r}}{T}+o_p(1), \label{eqn:1st-term-in-R1}\\
  \hat{\mathbf{Z}}_u^\top\mathbf{W}_e(\hat{\mathbf{e}}_T-\hat{\mathbf{Z}}_u\boldsymbol{\alpha}^*(u))&=\mathbf{Z}_u^\top\mathbf{W}_e
        (\mathbf{e}_T-\mathbf{Z}_u\boldsymbol{\alpha}^*(u))+o_p(1). \label{eqn:2nd-term-in-R1}
\end{align}
Then we have
\begin{align*}
R_1 &\geq
        \frac{\mathbf{r}^\top{\mathbf{Z}}_u^\top\mathbf{W}_e{\mathbf{Z}}_u\mathbf{r}}{T}
        -2\frac{\mathbf{r}^\top}{\log(1/h_e)}\sqrt{\frac{h_e}{T}}\mathbf{Z}_u^\top\mathbf{W}_e
        (\mathbf{e}_T-\mathbf{Z}_u\boldsymbol{\alpha}^*(u)) \\
      &~~~~+\sum_{k\in\mathcal{S}_1}\frac{h_e\lambda}{\log(1/h_e)w_k}\Big(\Big|\phi^*_k(u)+c_Tr_k\Big|-|\phi^*_k(u)|\Big) \\
      &~~~~+\sum_{k\in\mathcal{S}_2}\frac{\gamma}{\log(1/h_e)w'_k}\Big(\Big|h_e\phi^{*'}_k(u)+c_Tr_{p+k}\Big|-|\phi^{*'}_k(u)|\Big)+o_p(1), \\
    &
\end{align*}
Let $l_0^{\min}=\l^{\min}\big(\boldsymbol\Lambda(u)\big)$ and $l_T^{\min}=\inf_{u\in(0, 1)}\l^{\min}\big(\frac{{\mathbf{Z}}_u^\top\mathbf{W}_e{\mathbf{Z}}_u}{T}\big)$, where $\l^{\min}(\mathbf{A})$ denotes the minimal eigenvalue of matrix $\mathbf{A}$, and
\begin{equation*}
  \boldsymbol\Lambda(u) = \begin{pmatrix}
     \boldsymbol\Gamma(u) & \mathbf{O}_p \\
    \mathbf{O}_p & \mu_2\boldsymbol\Gamma(u)
  \end{pmatrix}.
\end{equation*}
Then
\begin{align*}
  R_1 &\geq \Vert r\Vert^2l_T^{\min}-2\Vert r\Vert\cdot\frac{1}{\sqrt{\log(1/h_e)}}\sup_u\Bigg\Vert\sqrt{\frac{h_e}{T}}\mathbf{Z}_u^\top\mathbf{W}_e\big(\mathbf{e}_T-\mathbf{Z}_u\boldsymbol{\alpha^*}(u)\big)\Bigg\Vert\\
  &~~~~ - \frac{h_e\lambda}{\sqrt{Th_e\log(1/h_e)}}\cdot\frac{\sqrt{p}}{\min_{k\in\mathcal{S}_1}w_k}\Vert r\Vert - \frac{\gamma}{\sqrt{Th_e\log(1/h_e)}}\cdot\frac{\sqrt{p}}{\min_{k\in\mathcal{S}_2}w'_k}\Vert r\Vert+o_p(1) \\
  &= l_T^{\min}\cdot C^2-\Bigg\{\frac{2}{\sqrt{\log(1/h_e)}}\sup_u\Bigg\Vert\sqrt{\frac{h_e}{T}}\mathbf{Z}_u^\top\mathbf{W}_e\big(\mathbf{e}_T-\mathbf{Z}_u\boldsymbol{\alpha^*}(u)\big)\Bigg\Vert\\
  &~~~~ - \frac{h\lambda}{\sqrt{Th_e\log(1/h_e)}}\cdot\frac{\sqrt{2p}}{\min_{k\in\mathcal{S}_1}w_k} - \frac{\gamma}{\sqrt{Th_e\log(1/h_e)}}\cdot\frac{\sqrt{2p}}{\min_{k\in\mathcal{S}_2}w'_k}\Bigg\}\cdot C+o_p(1) \\
  &\equiv l_T^{\min}\cdot C^2-A\cdot C+o_p(1).
\end{align*}
where
\begin{align*}
  A &= \frac{2}{\sqrt{\log(1/h_e)}}\sup_u\Bigg\Vert\sqrt{\frac{h_e}{T}}\mathbf{Z}_u^\top\mathbf{W}_e\big(\mathbf{e}_T-\mathbf{Z}_u\boldsymbol{\alpha^*}(u)\big)\Bigg\Vert- \frac{h\lambda}{\sqrt{Th_e\log(1/h_e)}}\cdot\frac{\sqrt{2p}}{\min_{k\in\mathcal{S}_1}w_k} \\
  &~~~~- \frac{\gamma}{\sqrt{Th_e\log(1/h_e)}}\cdot\frac{\sqrt{2p}}{\min_{k\in\mathcal{S}_2}w'_k}.
\end{align*}
Similar with \eqref{eqn:ztwz-limit}, we can show that
\begin{equation}\label{eqn:S-tilde-convergence}
  \frac{\mathbf{Z}_u^\top\mathbf{W}_e\mathbf{Z}_u}{T}~\stackrel{p}{\longrightarrow}~\boldsymbol\Lambda(u).
\end{equation}
Note that $\boldsymbol{\Gamma}(u)$ is nonsingular, it holds that $l_0^{\min}>0$. Hence, by the definition of $l_T^{\min}$ and $l_0^{\min}$ and  \eqref{eqn:S-tilde-convergence}, $l_T^{\min}~\stackrel{p}{\longrightarrow}~l_0^{\min}>0$.

On the other side, it follows from condition (C1) and Lemma 6.1 of \cite{fan2008nonlinear} that
\begin{equation}\label{eqn:second-term-in-R1}
  \sup_u\Bigg\Vert\sqrt{\frac{h_e}{T}}\mathbf{Z}_u^\top\mathbf{W}_e\big(\mathbf{e}_T-\mathbf{Z}_u\boldsymbol{\alpha^*}(u)\big)\Bigg\Vert=O_p(\sqrt{\log(1/h_e)}).
\end{equation}
Note that $\min_{k\in\mathcal{S}_1}w_k$ and $\min_{k\in\mathcal{S}_2}w'_k$ converge in probability to a positive constant respectively, which implies that
\begin{equation}\label{eqn:third-term-in-R1}
  \frac{h_e\lambda}{\sqrt{Th_e\log(1/h_e)}}\cdot\frac{\sqrt{2p}}{\min_{k\in\mathcal{S}_1}w_k}=o_p(1),~~~ \frac{\gamma}{\sqrt{Th_e\log(1/h_e)}}\cdot\frac{\sqrt{2p}}{\min_{k\in\mathcal{S}_2}w'_k}=o_p(1).
\end{equation}
So it follows from \eqref{eqn:second-term-in-R1} and \eqref{eqn:third-term-in-R1} that $A=O_p(1)$. Therefore, as long as the constant $C$ is large enough, the value of $l_T^{\min}\cdot C^2-A\cdot C$ is positive, which implies that $R_1>0$. Thus \eqref{eqn:fan&li-sufficient-condition} holds, which means that there exists a minimum in the ball $B_C$ for any $u\in(0, 1)$ with probability $1-\epsilon$. Therefore the minimizer $\hat{\boldsymbol{\alpha}}_{\lambda, \gamma}(u)$ of $Q(\boldsymbol{\alpha}(u)\big|\hat{e}, u)$ must satisfy that $\sup_u\Vert\hat{\boldsymbol{\alpha}}_{\lambda, \gamma}(u)-\boldsymbol{\alpha}^*(u)\Vert=O_p(c_T)$. The proof is completed.

\end{proof}

\begin{lemma}\label{lemma:uniform-sparsity-consistency-x-unavailable}
Assume that conditions (C1)-(C11) are satisfied. $\Gamma(u)$ defined in \eqref{eqn:Gamma} is nonsingular for all $u\in(0, 1)$ and has uniformly bounded second derivatives. When $h_e\lambda/\sqrt{Th_e}\to 0$,  $\gamma/\sqrt{Th_e}\to 0$, $T^{-1/5}(\log(T))^{-1/2}\lambda\to\infty$, and $T^{-1/5}(\log(T))^{-1/2}\gamma\to\infty$ as $T\to\infty$, we have
\begin{enumerate}[(i)]
  \item $P\Big(\sup_u\vert \hat{\boldsymbol{\phi}}_{\lambda, \mathcal{S}_1^c}(u)\vert=0, \sup_u\vert \hat{\boldsymbol{\phi}}^{'}_{\lambda, \mathcal{S}_2^c}(u)\vert=0\Big)\to 1$;
  \item $  P(\hat{\mathcal{S}}_{1\lambda}=\mathcal{S}_1, \hat{\mathcal{S}}_{2\gamma}=\mathcal{S}_2)\to1$.
\end{enumerate}
\end{lemma}
\begin{proof}
 The proof for (i) is similar to  Theorem 2 in \cite{wang2012parametric}. One difference is that the equations contain $\hat{e}_t$ can be easily handled by using equation \eqref{eqn:1st-term-in-R1} and \eqref{eqn:2nd-term-in-R1}, another is that Theorem 1 of \cite{wang2012parametric} are replaced by Lemma \ref{lemma:uniform-consistency-x-unavailable} in our case.

 The consistency of $\hat{\mathcal{S}}_{1\lambda}$ and $\hat{\mathcal{S}}_{2\gamma}$ in (ii) is implied from (i).
\end{proof}

\begin{theorem}\label{thm:BIC-consistency}
Assume conditions (C1)-(C11) are satisfied. $\Gamma(u)$ defined in \eqref{eqn:Gamma} is nonsingular for all $u\in(0, 1)$ and has uniformly bounded second derivatives, then it holds that
\begin{equation*}
  P(\hat{\mathcal{S}}_{1\hat\lambda}=\mathcal{S}_1, \hat{\mathcal{S}}_{2\hat\gamma}=\mathcal{S}_2)\to 1, ~~~~~~\mathrm{as}~T\to\infty.
\end{equation*}
\end{theorem}
\begin{proof} The proof is similar with the i.i.d. case in \cite{wang2012parametric}, we show the procedure briefly and omit the details. Denote $\hat{\alpha}_{\lambda, \gamma,k}(u)$ as the $k$th component of $\hat{\boldsymbol\alpha}_{\lambda, \gamma}(u)$. Let $\mathcal{S} = \{1, \cdots, p_2, p_2+1, \cdots, p_1\}$ and $\hat{\mathcal{S}}_{\lambda,\gamma} = \{k:\sum_{t=1}^T\vert\hat{\alpha}_{\lambda,\gamma,k}(t/T)\vert>0\}$. Then Theorem \ref{thm:BIC-consistency} is equivalent to $P(\hat{\mathcal{S}}_{\hat\lambda,\hat\gamma}=\mathcal{S})\to 1$. By (ii) of Lemma \ref{lemma:uniform-sparsity-consistency-x-unavailable}, we know that as $T\to\infty$, $\lambda_T=\gamma_T=T^{1/5}\log(T)$ satisfies that $P(\hat{\mathcal{S}}_{1\lambda_T}=\mathcal{S}_1, \hat{\mathcal{S}}_{2\gamma_T}=\mathcal{S}_2)\to 1$.

Then, we can divide $\mathbb{R}^2$ into three sets, i.e., $\mathbb{R}^2_{+}=\{(\lambda,\gamma): \mathcal{S}\subsetneq\mathcal{S}_{\lambda,\gamma}\}$, $\mathbb{R}^2_{-}=\{(\lambda,\gamma): \mathcal{S}\not\subset\mathcal{S}_{\lambda,\gamma}\}$, and $\mathbb{R}^2_{0}=\{(\lambda,\gamma): \mathcal{S}=\mathcal{S}_{\lambda,\gamma}\}$, corresponding to overfitted case, underfitted case, and correctly fitted case, respectively. The rest of the proof can be finished by mimicking the proof of Theorem 4 in \cite{wang2012parametric}, where Theorem 3 of in \cite{wang2012parametric} are replaced by Lemma \ref{lemma:uniform-sparsity-consistency-x-unavailable} in our case.
\end{proof}

In practice some data driven methods like grid search can be used to determine the tuning parameters $\lambda, \gamma$. Actually, $\lambda=\gamma$ is also allowed and it can save much computation cost according to the condition of Lemma \ref{lemma:uniform-sparsity-consistency-x-unavailable}. Moreover, we can set $\gamma=0$ in \eqref{eqn:penalized-least-squares} if we only focus on the purpose of identifying the nonzero coefficients.

\subsection{Asymptotic results on the refined estimator of $g(\cdot,\cdot)$}
With the estimator of the error structure, $\hat{\boldsymbol\phi}(u)$, Li and You \cite{liyouJASM} derived a refined estimator of $g(\cdot,\cdot)$, which is denoted as $\hat{g}^*(u,x)$. Theorem \ref{thm:trend-asymptotic-normality-refined} presents the asymptotic property of $\hat{g}^*(u,x)$ and show that $\hat{g}^*(u,x)$ is more efficient than the preliminary estimator $\hat{g}(u,x)$.
\begin{theorem}\label{thm:trend-asymptotic-normality-refined}
Let $h^*=O_p(T^{-\frac{1}{6}})$. Under the conditions (C1)-(C11), $r=\min\{\rho, 1\}>\frac{1}{2}$, it holds for any $u\in(0, 1)$ that
\begin{align*}
&\sqrt{Th^{*2}}
    \begin{bmatrix}
      \begin{pmatrix}
      \hat{g}^*(u,x) \\
      h^*\frac{\partial{\hat{g}^*(u,x)}}{\partial{u}} \\
      h^*\frac{\partial{\hat{g}^*(u,x)}}{\partial{x}}
      \end{pmatrix}
      -
      \begin{pmatrix}
      {g}(u,x) \\
      h^*\frac{\partial{{g}(u,x)}}{\partial{u}} \\
      h^*\frac{\partial{{g}(u,x)}}{\partial{x}}
      \end{pmatrix}
    -
    \begin{pmatrix}
    \frac{\mu_2h^{*2}}{2}\Big(\frac{\partial^2{g(u,x)}}{\partial^2{u}}+\frac{\partial^2{g(u,x)}}{\partial^2{x}}\Big) \\
    0 \\
    0
    \end{pmatrix}+o_p(h^{*2})
   \end{bmatrix}\\
    \stackrel{D}{\longrightarrow}&~N(\mathbf{0}, \mathbf{V}^*_{u,x}),~~~~~~~~~~\hspace{6cm}~\mathrm{ as}~T\to\infty.
\end{align*}
where  $\mathbf{V}^*_{u,x}=\frac{\sigma^2}{f(u, x)}\boldsymbol{\Sigma}$ and $\boldsymbol{\Sigma}$ is defined in \eqref{eqn:Sigma}.
%\sqrt{Th^2}\Big[\hat{g}^*(u, x)-g(u, x)-\frac{h^2}{2}\mu_2\Big(\frac{\partial^2{g(u,x)}}{\partial{u}^2}+\frac{\partial^2{g(u,x)}}{\partial{x}^2}\Big)
%  +o(h^2)\Big]\xlongrightarrow{D}\mathcal{N}(\mathbf{0}, V^*_{u,x}),
%where  $V_{u,x}=\nu_0^2\sigma^2/f(u, x)$.
\end{theorem}
\begin{proof} By model \eqref{model:ourmodel}, it is easy to have
\begin{align}
\hat{Y}_t^*&=Y_t-\sum_{k=1}^p\hat{\phi}_k(t/T)\big(Y_{t-k}-\hat{g}((t-k)/T, X_{t-k})\big) \notag\\
   &= g(t/T, X_{t})+\epsilon_t+\sum_{k=1}^p\big({\phi}_k(t/T)-\hat{\phi}_k(t/T)\big)e_{t-k}+\sum_{k=1}^p\hat{\phi}_k(t/T)\big\{\hat{g}((t-k)/T, X_{t-k}) \notag\\
   &~~~~-g((t-k)/T, X_{t-k})\big\}. \label{eqn:Yt-hat-star}
\end{align}
Let $\mathbf{M}_T=\frac{1}{T}\mathbf{Z}^*(u,x)^{\top}\mathbf{W}^*(u,x)\mathbf{Z}^{*}(u,x)$. It follows from \eqref{eqn:Yt-hat-star} that
\begin{align*}
\begin{pmatrix}
      \hat{g}^*(u,x) \\
      h_e\frac{\partial{\hat{g}^*(u,x)}}{\partial{u}} \\
      h_e\frac{\partial{\hat{g}^*(u,x)}}{\partial{x}}
    \end{pmatrix}
  &= \mathbf{M}_T^{-1}
        \begin{pmatrix}
      \frac{1}{T}\sum_{t=1}^TK_{h^*}(\frac{t}{T}-u)K_{h^*}(X_t-x)g(\frac{t}{T},X_t) \\
      \frac{1}{T}\sum_{t=1}^TK_{h^*}(\frac{t}{T}-u)K_{h^*}(X_t-x)\big(\frac{\frac{t}{T}-u}{h^*}\big)g(\frac{t}{T},X_t) \\
      \frac{1}{T}\sum_{t=1}^TK_{h^*}(\frac{t}{T}-u)K_{h^*}(X_t-x)(\frac{X_t-x}{h^*})g(\frac{t}{T},X_t)
    \end{pmatrix} \\
  &~~+ \mathbf{M}_T^{-1}
        \begin{pmatrix}
      \frac{1}{T}\sum_{t=1}^TK_{h^*}(\frac{t}{T}-u)K_{h^*}(X_t-x)\epsilon_t \\
      \frac{1}{T}\sum_{t=1}^TK_{h^*}(\frac{t}{T}-u)K_{h^*}(X_t-x)\big(\frac{\frac{t}{T}-u}{h^*}\big)\epsilon_t \\
      \frac{1}{T}\sum_{t=1}^TK_{h^*}(\frac{t}{T}-u)K_{h^*}(X_t-x)(\frac{X_t-x}{h^*})\epsilon_t
    \end{pmatrix} \\
  &~~+ \mathbf{M}_T^{-1}
        \begin{pmatrix}
      \frac{1}{T}\sum_{t=1}^TK_{h^*}(\frac{t}{T}-u)K_{h^*}(X_t-x)\sum_{k=1}^p\big({\phi}_k(\frac{t}{T})-\hat{\phi}_k(\frac{t}{T})\big)e_{t-k} \\
      \frac{1}{T}\sum_{t=1}^TK_{h^*}(\frac{t}{T}-u)K_{h^*}(X_t-x)\big(\frac{\frac{t}{T}-u}{h^*}\big)\sum_{k=1}^p\big({\phi}_k(\frac{t}{T})-\hat{\phi}_k(\frac{t}{T})\big)e_{t-k} \\
      \frac{1}{T}\sum_{t=1}^TK_{h^*}(\frac{t}{T}-u)K_{h^*}(X_t-x)(\frac{X_t-x}{h^*})\sum_{k=1}^p\big({\phi}_k(\frac{t}{T})-\hat{\phi}_k(\frac{t}{T})\big)e_{t-k}
    \end{pmatrix} \\
  &~~+ \mathbf{M}_T^{-1}
        \begin{pmatrix}
      \frac{1}{T}\sum_{t=1}^TK_{h^*}(\frac{t}{T}-u)K_{h^*}(X_t-x)\Delta \\
      \frac{1}{T}\sum_{t=1}^TK_{h^*}(\frac{t}{T}-u)K_{h^*}(X_t-x)\big(\frac{\frac{t}{T}-u}{h^*}\big)\Delta \\
      \frac{1}{T}\sum_{t=1}^TK_{h^*}(\frac{t}{T}-u)K_{h^*}(X_t-x)(\frac{X_t-x}{h^*})\Delta
    \end{pmatrix} \\
  &\stackrel{\vartriangle}{=} J_1+J_2+J_3+J_4,
\end{align*}
where $\Delta=\sum_{k=1}^p\hat{\phi}_k(t/T)\big[\hat{g}\big(\frac{t-k}{T}, X_{t-k}\big)-g\big(\frac{t-k}{T}, X_{t-k}\big)\big]$.

\noindent(i) Following the same way as the proof of Theorem \ref{thm:trend-asymptotic-normality}, we can easily show that
\begin{equation*}
  \sqrt{Th^*}\begin{bmatrix}
      J_1-
      \begin{pmatrix}
      {g}(u,x) \\
      h\frac{\partial{{g}(u,x)}}{\partial{u}} \\
      h\frac{\partial{{g}(u,x)}}{\partial{x}}
      \end{pmatrix}
    \end{bmatrix}
    =
    \begin{pmatrix}
    \frac{\mu_2h_e^2}{2}\Big(\frac{\partial^2{g(u,x)}}{\partial^2{u}}+\frac{\partial^2{g(u,x)}}{\partial^2{x}}\Big) \\
    0 \\
    0
    \end{pmatrix}+o_p(h^{*2}).
\end{equation*}

\noindent(ii) By Theorem \ref{thm:alpha-asymptotic-normality} and the proof of Theorem \ref{thm:trend-asymptotic-normality}, it follows that
\begin{equation*}
  \sqrt{Th^*}J_2\stackrel{D}{\longrightarrow}~N(\mathbf{0},\mathbf{V}^*_{u,x}), ~~~~~~~~\mathrm{as}~T\to\infty,
\end{equation*}
where $\mathbf{V}^*_{u,x}=\frac{\sigma^2}{f(u, x)}\boldsymbol{\Sigma}$ and $\boldsymbol{\Sigma}$ is defined in \eqref{eqn:Sigma}.

\noindent(iii) Theorem \ref{thm:alpha-asymptotic-normality} implies that
$  \Vert\hat{\boldsymbol\phi}(u)-\boldsymbol\phi(u)\Vert=O_p(h_e^2+\frac{1}{Th_e})$. So it holds that
\begin{align*}
  &~~~~~\Bigg\Vert\mathbf{M}_T^{-1}
        \begin{pmatrix}
      \frac{1}{T}\sum_{t=1}^TK_{h^*}(\frac{t}{T}-u)K_{h^*}(X_t-x)\sum_{k=1}^p\big({\phi}_k(\frac{t}{T})-\hat{\phi}_k(\frac{t}{T})\big)e_{t-k} \\
      \frac{1}{T}\sum_{t=1}^TK_{h^*}(\frac{t}{T}-u)K_{h^*}(X_t-x)\big(\frac{\frac{t}{T}-u}{h^*}\big)\sum_{k=1}^p\big({\phi}_k(\frac{t}{T})-\hat{\phi}_k(\frac{t}{T})\big)e_{t-k} \\
      \frac{1}{T}\sum_{t=1}^TK_{h^*}(\frac{t}{T}-u)K_{h^*}(X_t-x)(\frac{X_t-x}{h^*})\sum_{k=1}^p\big({\phi}_k(\frac{t}{T})-\hat{\phi}_k(\frac{t}{T})\big)e_{t-k}
    \end{pmatrix}\Bigg\Vert \\
  &\leq\Bigg\Vert\mathbf{M}_T^{-1}
        \begin{pmatrix}
      \frac{1}{T}\sum_{t=1}^TK_{h^*}(\frac{t}{T}-u)K_{h^*}(X_t-x) \\
      \frac{1}{T}\sum_{t=1}^TK_{h^*}(\frac{t}{T}-u)K_{h^*}(X_t-x)\big(\frac{\frac{t}{T}-u}{h^*}\big) \\
      \frac{1}{T}\sum_{t=1}^TK_{h^*}(\frac{t}{T}-u)K_{h^*}(X_t-x)(\frac{X_t-x}{h^*})
    \end{pmatrix}\Bigg\Vert \max_{t\in\{p+1,\cdots,T\}}\Vert\sum_{k=1}^p\big({\phi}_k(\frac{t}{T})-\hat{\phi}_k(\frac{t}{T})\big)e_{t-k}\Vert \\
  &=O_p(h_e^2).
\end{align*}
Therefore, $J_3=o_p\bigg(h^{*2}+\sqrt{\frac{\log T}{Th^{*2}}}\bigg)$.

\noindent(iv) Condition (C8) implies that there exist some constant $C$ such that $|\phi_k(t/T)|\leq C<\infty$ for all $k=1,\cdots,p$, then
\begin{align*}
  \Delta &\leq p\times\max_{k\in\{1,\cdots,p\}}\{\phi_k(t/T)+o_p(1)\}\times \sup_{u,x}|\hat{g}(u,x)-g(u,x)| \\
  &=p[C+o_p(1)]\times O_p\Big(\sqrt{\frac{\log T}{Th^2}}+h^2\Big) \\
  &=O_p\Big(\sqrt{\frac{\log T}{Th^2}}+h^2\Big).
\end{align*}
By the standard results from density estimation, it holds that $J_4=o_p(h^{*2})$.

Combining the results of (i)-(iv), we have that
\begin{align*}
&~~~~~\sqrt{Th^{*2}}
    \begin{bmatrix}
      \begin{pmatrix}
      \hat{g}^*(u,x) \\
      h^*\frac{\partial{\hat{g}^*(u,x)}}{\partial{u}} \\
      h^*\frac{\partial{\hat{g}^*(u,x)}}{\partial{x}}
      \end{pmatrix}
      -
      \begin{pmatrix}
      {g}(u,x) \\
      h^*\frac{\partial{{g}(u,x)}}{\partial{u}} \\
      h^*\frac{\partial{{g}(u,x)}}{\partial{x}}
      \end{pmatrix}
    -
    \begin{pmatrix}
    \frac{\mu_2h^{*2}}{2}\Big(\frac{\partial^2{g(u,x)}}{\partial^2{u}}+\frac{\partial^2{g(u,x)}}{\partial^2{x}}\Big) \\
    0 \\
    0
    \end{pmatrix}+o_p(h^{*2})
   \end{bmatrix} \\
  &=\sqrt{Th^{*2}}
    \begin{bmatrix}
      J_1+J_2-
      \begin{pmatrix}
      {g}(u,x) \\
      h^*\frac{\partial{{g}(u,x)}}{\partial{u}} \\
      h^*\frac{\partial{{g}(u,x)}}{\partial{x}}
      \end{pmatrix}
    -
    \begin{pmatrix}
    \frac{\mu_2h^{*2}}{2}\Big(\frac{\partial^2{g(u,x)}}{\partial^2{u}}+\frac{\partial^2{g(u,x)}}{\partial^2{x}}\Big) \\
    0 \\
    0
    \end{pmatrix}+o_p(h^{*2})
   \end{bmatrix} \\
  &\stackrel{D}{\longrightarrow}~N(\mathbf{0}, \mathbf{V}^*_{u,x}),~~~~~~~~~~~~\mathrm{as}~T\to\infty.
\end{align*}
\end{proof}

\begin{remark} There are two bandwidths $h$ and $h^*$ in the procedure of estimating $g(u,x)$. Theorem \ref{thm:trend-asymptotic-normality-refined} shows that the bandwidth $h^*$ in the refined estimator should be of the standard order of estimating a binary nonparametric function. However, the bandwidth $h$ for the preliminary estimator $\hat{g}(u, x)$ should be of smaller order $h=o(h^*)$ to control the bias in the first step of the estimation. Specially, to ensure the optimal convergence rate of the estimators of the error autoregressive structure as shown in Theorem \ref{thm:alpha-asymptotic-normality}, the order of $h$ should be even smaller than the bandwidth $h_e$, i.e., $h=o(T^{-1/5})$. In practice, standard bandwidth selection methods can be utilized for $h^*$ and $h_e$. Then $h^*$ can be multiplied by a constant like 0.5 to obtain $h$ as suggested by \cite{liu2010nonparametric}.
\end{remark}
%%%%%%%%%%%%%%%%%%%%%%%%%%%%%%%%%%%%%%%%%%%%%%%%%%%%%%%%%%%%%%%%%%%%%%%%%%%%%%%%%%%%%%%%%%%%%%%%%%%%%%%%%%%%%%%
\section{Numerical studies}
 \cite{liyouJASM} conducted two real data studies for their proposed method. In order to give an overall evaluation on the finite sample performance of their proposed method, two simulation studies are conducted in this section. The Epanechnikov kernel $K(t)=0.75(1-t^2)_{+}$ is used throughout this section. In each study, the leave-one-out cross-validation is applied to select the optimal bandwidth for $h^*$ and $h_e$, and $h=0.5h^*$ is for the preliminary estimator. Since the results are not very sensitive to the bandwidth, only the case of the optimal bandwidth is reported here.

The first study is designed to compare the performance of the preliminary estimator and the refined estimator of the time-varying nonparametric function.

\begin{example}\label{ex:1} Assume that the order of the error autoregressive structure is known. $Y_t$ is generated by the model
\begin{equation*}
   Y_t = g(t/T, X_t)+e_t, ~~~~~~~~~~~~~t=1, \cdots, T,
\end{equation*}
where $g(u, x)=1.5\cos(2\pi u)x^2$, and the explanatory variable $X_t$ is generated from a time-varying AR(1) process that is locally stationary, i.e.,
\begin{equation*}
  X_t = 0.7t/TX_{t-1}+0.5\xi_t,
\end{equation*}
where $\xi_t\sim N(0, 1)$. The error term $e_t$ is generated by the time-varying autoregressive process $e_t-\sum_{i=1}^5\phi_i(t/T)e_{t-i}=\epsilon_{t}$, where $\epsilon_t$ follows normal distribution $\mathcal{N}(0, \sigma^2)$, and the time-varying autoregressive coefficient function $\boldsymbol\phi(u)=(\phi_1(u),\phi_2(u),\phi_3(u),\phi_4(u),\phi_5(u))^{\top}$ are specified as follows.
\begin{enumerate}[Model (a)]
  \item $\boldsymbol\phi(u) = ((-0.1+0.6\sin(2\pi u), 0, 0, 0, 0)^\top$;
  \item $\boldsymbol\phi(u) = (3(u-0.4)^2-0.6, 0.3, 0, 0, 0)^\top$;
  \item $\boldsymbol\phi(u) = (5(u-0.5)^2-0.6, -1+\sin^2(\pi u), 0, 0, 0)^\top$.
\end{enumerate}
\end{example}

Moreover, the square root of average squared error (RASE) criterion is used to evaluate the performance of the estimators. For an estimator $\hat{g}(u, x)$, its RASE is defined as
\begin{equation*}
  \RASE\big(\hat{g}(u, x)\big) = \bigg[\frac{1}{T}\sum_{t=1}^T\Big(\hat{g}\big(\frac{t}{T}, X_t\big)-g\big(\frac{t}{T}, X_t\big)\Big)^2\bigg]^{1/2}.
\end{equation*}

With $T=200,300,400$, $\sigma=0.5, 1$, and $\boldsymbol\phi(u)$ being model (a)-(c), the empirical mean values and standard deviations (SD) of RASE based on 500 replications are presented in Table \ref{tab:RASE-trend}. For the comparison purpose, the oracle estimator $\hat{g}^{\OR}(u,x)$, which is the local linear estimator of $g(u, x)$ when the time-varying autoregressive error structure is completely known, is also calculated and presented in Table \ref{tab:RASE-trend}.
\begin{table}
\begin{tabular}{llccccccc}
\toprule
$\boldsymbol\phi(u)$    & $\sigma$ & $T$ & \multicolumn{2}{c}{$\hat{g}(u, x)$} &   \multicolumn{2}{c}{$\hat{g^*}(u, x)$} &   \multicolumn{2}{c}{$\hat{g}^{\OR}(u, x)$} \\
\cmidrule(r){4-5}\cmidrule(r){6-7}\cmidrule(r){8-9}
    & &     &   Mean &    SD &   Mean &    SD &   Mean &    SD \\
\midrule
\multirow{4}{1in}{Model (a)} & 0.5 & 200 &  0.253 &  0.029 &  0.166 &  0.027 &  0.161 &  0.024 \\
&    & 300 &  0.208 &  0.024 &  0.145 &  0.022 &  0.142 &  0.021 \\
&    & 400 &  0.182 &  0.020 &  0.136 &  0.019 &  0.135 &  0.018 \\
\cmidrule(r){2-9}
& 1.0 & 200 &  0.505 &  0.058 &  0.302 &  0.052 &  0.290 &  0.046 \\
&    & 300 &  0.415 &  0.047 &  0.248 &  0.042 &  0.242 &  0.037 \\
&    & 400 &  0.362 &  0.040 &  0.222 &  0.035 &  0.219 &  0.031 \\
\midrule
\multirow{4}{1in}{Model (b)} & 0.5 & 200 &  0.312 &  0.048 &  0.189 &  0.040 &  0.161 &  0.024 \\
&    & 300 &  0.260 &  0.041 &  0.165 &  0.035 &  0.142 &  0.021 \\
&    & 400 &  0.231 &  0.031 &  0.154 &  0.027 &  0.135 &  0.018 \\
\cmidrule(r){2-9}
& 1.0 & 200 &  0.624 &  0.096 &  0.351 &  0.080 &  0.290 &  0.046 \\
&    & 300 &  0.519 &  0.081 &  0.292 &  0.070 &  0.242 &  0.037 \\
&    & 400 &  0.459 &  0.062 &  0.262 &  0.055 &  0.219 &  0.032 \\
\midrule
\multirow{4}{1in}{Model (c)} & 0.5 & 200 &  0.366 &  0.079 &  0.193 &  0.047 &  0.161 &  0.024 \\
&    & 300 &  0.322 &  0.074 &  0.162 &  0.035 &  0.142 &  0.021 \\
&    & 400 &  0.298 &  0.070 &  0.149 &  0.029 &  0.135 &  0.018 \\
\cmidrule(r){2-9}
& 1.0 & 200 &  0.732 &  0.159 &  0.359 &  0.098 &  0.290 &  0.046 \\
&    & 300 &  0.643 &  0.148 &  0.286 &  0.073 &  0.242 &  0.037 \\
&    & 400 &  0.593 &  0.141 &  0.251 &  0.061 &  0.219 &  0.032 \\
\bottomrule
\end{tabular}
\caption{Means and SDs of the RASEs of the estimators for $g(u, x)=1.5\cos(2\pi u)x^2$}
\label{tab:RASE-trend}
\end{table}

From Table \ref{tab:RASE-trend}, we can draw the following conclusions:
\begin{itemize}
  \item Under all of three error term models, the refined estimator $\hat{g}^*(u, x)$  has both smaller mean values and standard deviations of RASE than the preliminary estimator $\hat{g}(u, x)$. The improvement of $\hat{g}^*(u, x)$ is getting more significant as the complexity of the error term increases.
  \item An increase in $T$ results in a decrease in the mean values and standard deviations of RASE for all estimators. An increase in $\sigma$ results in an increase in the mean values and standard deviations of RASE for all estimators.
  \item The performance of the refined estimator $\hat{g}^*(u, x)$ is very close to $\hat{g}^{\OR}(u,x)$, which is also consistent with the theoretical results.
\end{itemize}

Figure \ref{fig:trend-boxplot} depicts the boxplots of $\hat{g}(u,x)$ and $\hat{g}^*(u,x)$ at points $(u,x)=(0.2, -0.5), (0.5, 0), (0.75, 0.4)$ under Model (b). It shows that the standard error is decreasing with the increase of $T$, and the bias is negligible when $T$ is large, which is claimed by Theorem \ref{thm:trend-asymptotic-normality-refined}.

\begin{figure}
  \centering
  \includegraphics[width=0.9\textwidth]{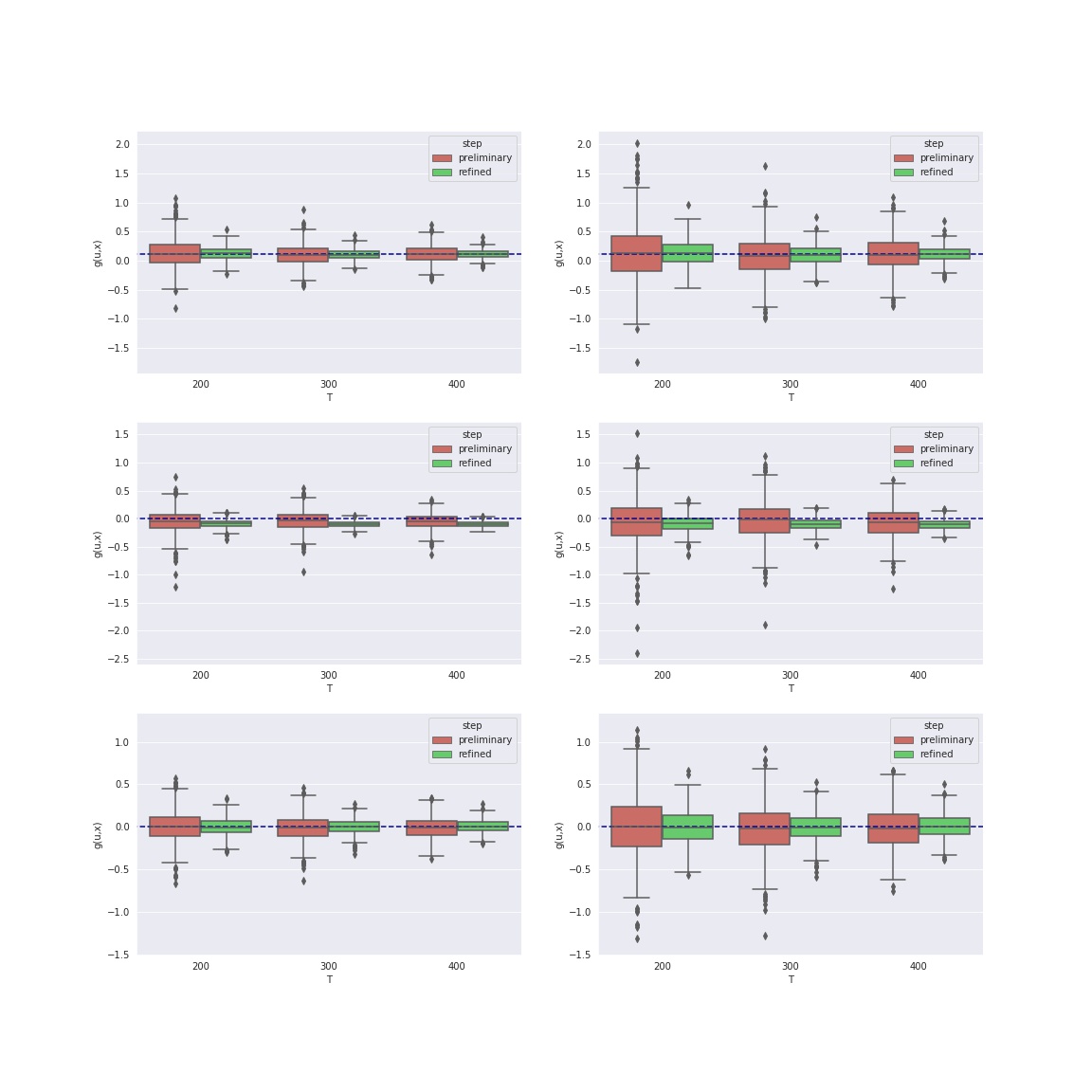}
  \caption{Boxplots of $\hat{g}(u,x)$ and $\hat{g}^*(u,x)$ for Model (b). The left panel: $\sigma=0.5$, the right panel: $\sigma=1$. The top panel: $(u, x) = (0.2, -0.5)$, middle panel: $(u, x) = (0.5, 0)$, bottom panel: $(u, x) = (0.75, 0.4)$}\label{fig:trend-boxplot}
\end{figure}

To further illustrate the performance of the refined estimator, we plot the average refined estimated function, the true function, and the bias of the average estimated function in Figure \ref{fig:averaged-estimate} for case (a) with $T=300$ and $\sigma=0.5$. The estimated surface is consistent with the true surface, which validate the theoretical results. We omit the results for other cases, which are similar.
\begin{figure}
  \centering
  \includegraphics[width=0.9\textwidth]{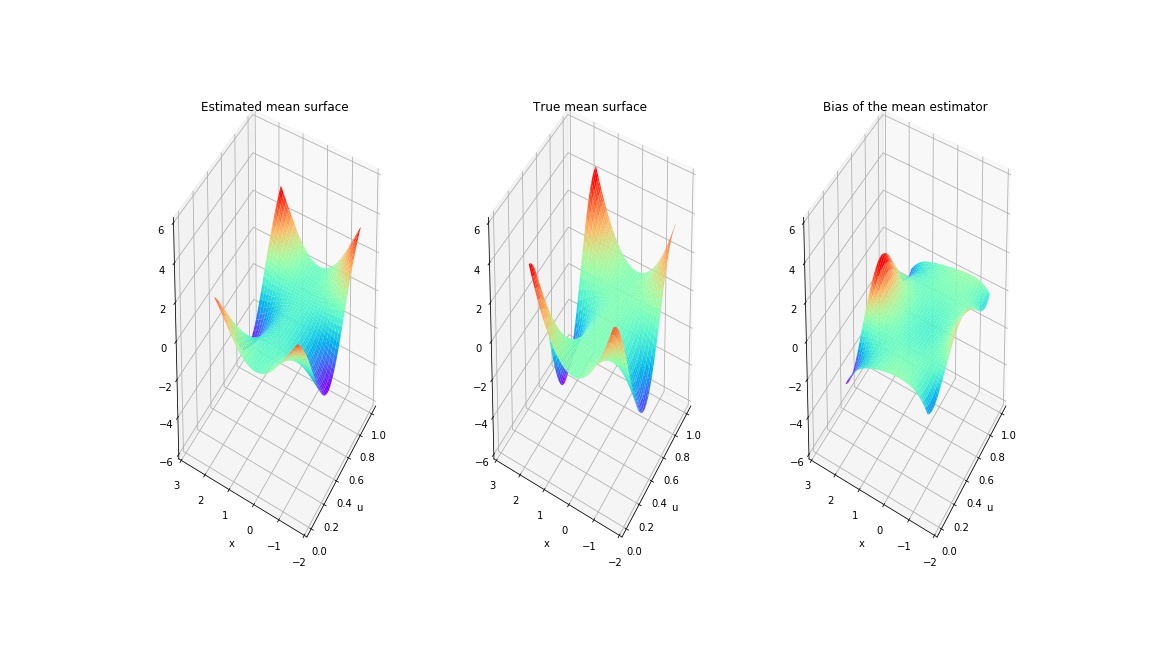}
  \caption{Averaged refined estimated function, true function, and bias of the average estimated function}\label{fig:averaged-estimate}
\end{figure}

\begin{example}
In this example, we aim to show the performance of the ULASSO estimator. The model settings and simulation settings are the same as in Example \ref{ex:1} except that the order of the error autoregressive structure is unknown. It's obvious that Model (a) is a tvAR(1) with $\mathcal{S}_1=\mathcal{S}_2=\{1\}$; Model (b) is a tvAR(2) with $\mathcal{S}_1=\{1, 2\}$ and $\mathcal{S}_2=\{1\}$; Model (c) is a tvAR(2) with $\mathcal{S}_1=\mathcal{S}_2=\{1, 2\}$.

Under the assumption that the order of the error autoregressive structure is unknown, the ULASSO estimator is adopted to estimate the autoregressive structure. $\hat{\phi}_k(u)$ is denoted as the optimal ULASSO estimate of which the shrinkage parameters are determined by the BIC criterion \eqref{eqn:BIC-x-unavailable}. To evaluate the performance of ULASSO estimate, the
result of variable selection (VS, i.e. nonzero coefficient selection) is classified as \cite{wang2012parametric}: 1. underfitted (at least one true nonzero variable is missing); 2. correctly fitted; 3. overfitted (all the significant variables are identified while at least one spurious variable is included).
The percentages in each category are presented under the heading `VS' in Table \ref{tab:BIC}. Similarly, we can partition the results for detecting the true coefficient functions and constant coefficient into these categories. Therefore, a correct selection means that the procedure identified both types correctly, an underfitted indicating missing at least one true function or a coefficient etc. These percentages are given under the heading `VS \& CI' in Table  \ref{tab:BIC}.

From Table \ref{tab:BIC}, we can see that for every model, the percentage of correctly fitted models is satisfying and it increases steadily with the sample size. Furthermore, the percentage of correct selection (VS \& CI) is also acceptable, especially when the sample size is moderate or large. On the other hand the percentage of VS \& CI is slightly less than that of VS. This is not surprising since the convergence speed of the derivative is a little slower than that of the coefficient function.

\begin{table}
\centering
\begin{tabular}{llccccccc}
\toprule
$\boldsymbol\phi(u)$    & $\sigma$ & $T$ & \multicolumn{3}{c}{VS} & \multicolumn{3}{c}{VS $\&$ CI}\\
\cmidrule(r){4-6}\cmidrule(r){7-9}
    & &     &  Under &  Correct &   Over &  Under &  Correct &   Over \\
\midrule
\multirow{4}{1in}{Model (a)} & 0.5 & 200 &  0.050 &    0.872 &  0.078 &  0.182 &    0.734 &  0.084 \\
&    & 300 &  0.006 &    0.916 &  0.078 &  0.030 &    0.888 &  0.082 \\
&    & 400 &  0.000 &    0.962 &  0.038 &  0.008 &    0.952 &  0.040 \\
\cmidrule(r){2-9}
& 1.0 & 200 &  0.206 &    0.724 &  0.070 &  0.294 &    0.626 &  0.080 \\
&    & 300 &  0.048 &    0.880 &  0.072 &  0.064 &    0.856 &  0.080 \\
&    & 400 &  0.000 &    0.938 &  0.062 &  0.004 &    0.928 &  0.068 \\
\midrule
\multirow{4}{1in}{Model (b)} & 0.5 & 200 &  0.088 &    0.712 &  0.200 &  0.274 &    0.496 &  0.230 \\
&    & 300 &  0.024 &    0.802 &  0.174 &  0.092 &    0.692 &  0.216 \\
&    & 400 &  0.000 &    0.890 &  0.110 &  0.020 &    0.846 &  0.134 \\
\cmidrule(r){2-9}
& 1.0 & 200 &  0.114 &    0.672 &  0.214 &  0.314 &    0.454 &  0.232 \\
&    & 300 &  0.028 &    0.790 &  0.182 &  0.100 &    0.688 &  0.212 \\
&    & 400 &  0.000 &    0.870 &  0.130 &  0.026 &    0.820 &  0.154 \\
\midrule
\multirow{4}{1in}{Model (c)} & 0.5 & 200 &  0.172 &    0.480 &  0.348 &  0.208 &    0.404 &  0.388 \\
&    & 300 &  0.006 &    0.608 &  0.386 &  0.006 &    0.554 &  0.440 \\
&    & 400 &  0.002 &    0.694 &  0.304 &  0.004 &    0.650 &  0.346 \\
\cmidrule(r){2-9}
& 1.0 & 200 &  0.192 &    0.450 &  0.358 &  0.234 &    0.382 &  0.384 \\
&    & 300 &  0.006 &    0.570 &  0.424 &  0.008 &    0.528 &  0.464 \\
&    & 400 &  0.004 &    0.664 &  0.332 &  0.006 &    0.602 &  0.392 \\
\bottomrule
\end{tabular}
\caption{The percentage of underfitted / correctly fitted /overfitted fitted models selection.
VS: variable selection; CI: constant coefficient identification}
\label{tab:BIC}
\end{table}

To further evaluate the ULASSO estimate, the mean values and standard deviations of  RASE of $\hat{\phi}_k(u)$ are calculated and presented in Table \ref{tab:RASE}. Moreover, the performance of constant coefficient estimation for $\phi_2(\cdot)\equiv\phi_2$ in Model (b) is also evaluated. If it is identified as the constant, i.e., $2\notin\hat{\mathcal{S}}_2$, the estimate of $\phi_2$ is taken as the mean of estimate at each time point, i.e.,
\begin{equation*}
  \hat\phi_2 = \frac{1}{T}\sum_{t=1}^T\hat\phi_2(t/T).
\end{equation*}
The bias and the standard error (SE) of $\hat\phi_2$ are included in the last two columns of Table \ref{tab:RASE}. It shows from Table \ref{tab:RASE} that under all of three error term models, an increase in $T$ results in a decrease in the mean values and standard deviations of RASE for all estimates. An increase in $\sigma$ results in an increase in the mean values and standard deviations of RASE. Moreover, for the constant coefficient estimation $\hat{\phi}_2$ for Model (b), the bias and the standard error of $\hat\phi_2$ also decreases with the increase of $T$. From Table \ref{tab:RASE}, we can conclude that the performance of the estimates for both time-varying coefficient function and the non-zero constant coefficient are satisfying.

\begin{table}
\centering
\begin{tabular}{llccccccc}
\toprule
$\boldsymbol\phi(u)$     & $\sigma$ & $T$ & \multicolumn{4}{c}{RASE} & \multicolumn{2}{c}{estimate}\\
\cmidrule(r){4-7}\cmidrule(r){8-9}
    & &     &   \multicolumn{2}{c}{$\hat{\phi}_1(\cdot)$} &   \multicolumn{2}{c}{$\hat{\phi}_2(\cdot)$} &   \multicolumn{2}{c}{$\hat{\phi}_2$} \\
\cmidrule(r){4-5}\cmidrule(r){6-7}\cmidrule(r){8-9}
    & &     &   mean &    SD &   mean &    SD &   bias &    SE \\
\midrule
\multirow{4}{1in}{Model (a)} & 0.5 & 200 &  0.260 &  0.066 &    $-$ &    $-$ &    $-$ &    $-$ \\
&    & 300 &  0.223 &  0.047 &    $-$ &    $-$ &    $-$ &    $-$ \\
&    & 400 &  0.189 &  0.038 &    $-$ &    $-$ &    $-$ &    $-$ \\
\cmidrule(r){2-9}
& 1.0 & 200 &  0.279 &  0.089 &    $-$ &    $-$ &    $-$ &    $-$ \\
&    & 300 &  0.228 &  0.060 &    $-$ &    $-$ &    $-$ &    $-$ \\
&    & 400 &  0.205 &  0.034 &    $-$ &    $-$ &    $-$ &    $-$ \\
\midrule
\multirow{4}{1in}{Model (b)} & 0.5 & 200 &  0.200 &  0.047 &  0.136 &  0.075 &  0.094 &  0.107 \\
&    & 300 &  0.173 &  0.038 &  0.102 &  0.061 &  0.057 &  0.088 \\
&    & 400 &  0.158 &  0.031 &  0.075 &  0.042 &  0.028 &  0.066 \\
\cmidrule(r){2-9}
& 1.0 & 200 &  0.206 &  0.056 &  0.149 &  0.078 &  0.116 &  0.108 \\
&    & 300 &  0.173 &  0.040 &  0.109 &  0.064 &  0.071 &  0.088 \\
&    & 400 &  0.162 &  0.032 &  0.081 &  0.045 &  0.040 &  0.068 \\
\midrule
\multirow{4}{1in}{Model (c)} & 0.5 & 200 &  0.275 &  0.082 &  0.226 &  0.101 &    $-$ &    $-$ \\
&    & 300 &  0.211 &  0.054 &  0.155 &  0.049 &    $-$ &    $-$ \\
&    & 400 &  0.179 &  0.052 &  0.125 &  0.040 &    $-$ &    $-$ \\
\cmidrule(r){2-9}
& 1.0 & 200 &  0.280 &  0.084 &  0.240 &  0.110 &    $-$ &    $-$ \\
&    & 300 &  0.223 &  0.055 &  0.155 &  0.052 &    $-$ &    $-$ \\
&    & 400 &  0.194 &  0.054 &  0.127 &  0.041 &    $-$ &    $-$ \\
\bottomrule
\end{tabular}
\caption{Means and SDs of the RASEs of the estimator for $\boldsymbol\phi(u)$ and bias and SE of the constant coefficient estimation }\label{tab:RASE}
\end{table}

\end{example}
%%%%%%%%%%%%%%%%%%%%%%%%%%%%%%%%%%%%%%%%%%%%%%%%%%%%%%%%%%%%%%%%%%%%%%%%%%%%%%%%%%%%%%%%%%%%%%%%%%%%%%%%%%%%%%%%
\section{Conclusion}
In this paper, we focus on the theoretical results of statistical inference for a class nonparametric regression models with time-varying
regression function, local stationary regressors and time-varying AR(p) (tvAR(p)) error process \eqref{model:ourmodel},
for which the estimation procedure is studied by \cite{liyouJASM} with no theoretical discussion. With regard to their estimators, we establish the convergence rate and asymptotic normality of the preliminary estimation $\hat{g}(u,x)$ of nonparametric regression function, show that their estimator of error term $\hat{\phi}_k(u)$ has the same convergence rate and the asymptotic distribution as the estimator when $e_t$ is observable, present the asymptotic property of the refined estimation $\hat{g}^*(u,x)$ of nonparametric regression function, and show that $\hat{g}^*(u,x)$ is more efficient than the preliminary estimator $\hat{g}(u,x)$. In addition, with regard to the ULASSO method for variable selection and constant coefficient detection for error term structure, we show that the ULASSO estimator can identify the true error term structure consistently. Simulation results have been provided to illustrate the finite sample performance of the estimators and validate our theoretical discussion on the properties of the estimators.

The model introduced by \cite{liyouJASM} is with time-varying AR(p) (tvAR(p)) error process. Actually, it can be extended to the model with time-varying autoregressive moving average model (tvARMA(p,q)), which is also local stationary. The essential thought on this article can be also extended to the estimations under the model with tvARMA(p,q) error term.

%%%%%%%%%%%%%%%%%%%%%%%%%%%%%%%%%%%%%%%%%%%%%%%%%%%%%%%%%%%%%%%%%%%%%%%%%%%%%%%%%%%%%%%%%%%%%%%%%%%%%%%%%%%%%%%%
\section*{Acknowledgements}
We would like to express our gratitude to Dr. Jinhong You for his heuristic discussion on the article. Our thanks also go to the referees for their time and comments.

%\bibliographystyle{apalike}
%\bibliography{tvar}

\end{document}